\documentclass{amsart}\usepackage{}\usepackage{cases}\usepackage{amsmath}\usepackage{amssymb}\usepackage{amsfonts}
\newtheorem{theorem}{Theorem}[section]
\newtheorem{lemma}[theorem]{Lemma}
\newtheorem{corollary}[theorem]{Corollary}
\newtheorem{proposition}[theorem]{Proposition}

\theoremstyle{definition}

\newtheorem{remark}[theorem]{Remark}
\numberwithin{equation}{section}

\begin{document}
\title[Heat kernel on smooth metric measure spaces and applications]
{Heat kernel on smooth metric\\ measure spaces and applications}
\author{Jia-Yong Wu}
\address{Department of Mathematics, Shanghai Maritime University,
Haigang Avenue 1550, Shanghai 201306, P. R. China}
\address{Department of Mathematics, Cornell University, Ithaca, NY 14853,
United States} \email{jywu81@yahoo.com}
\author{Peng Wu}
\address{Department of Mathematics, Cornell University,
Ithaca, NY 14853, United States} \email{wupenguin@math.cornell.edu}
\thanks{}
\subjclass[2010]{Primary 35K08; Secondary 53C21, 58J35.}
\dedicatory{}
\date{\today}

\keywords{smooth metric measure space, Bakry-\'Emery Ricci
curvature, heat kernel, Harnack inequality, Liouville theorem,
eigenvalue, Green's function, parabolicity.}
\begin{abstract}
We derive a Harnack inequality for positive solutions of the $f$-heat
equation and Gaussian upper and lower bounds for the $f$-heat kernel
on complete smooth metric measure spaces $(M,g,e^{-f}dv)$ with
Bakry-\'Emery Ricci curvature bounded below. The lower bound is
sharp. The main argument is the De Giorgi-Nash-Moser theory. As
applications, we prove an $L_f^1$-Liouville theorem for
$f$-subharmonic functions and an $L_f^1$-uniqueness theorem for
$f$-heat equations when $f$ has at most linear growth. We also
obtain eigenvalues estimates and $f$-Green's function
estimates for the $f$-Laplace operator.
\end{abstract}
\maketitle


\section{Introduction}

Heat kernel estimate is one of the fundamental problems in
Riemannian geometry. For Riemannian manifolds, there are two
classical methods for the heat kernel estimate. One is the gradient
estimate technique developed by Li and Yau \cite{[Li-Yau1]}, using
which they derived two-sided Gaussian bounds for the heat kernel on
Riemannian manifolds with Ricci curvature bounded below. The other
is the Moser iteration technique invented by Moser \cite{[Moser]}.
Grigor'yan \cite{[Grig]} and Saloff-Coste \cite{[Saloff], [Saloff1],
[Saloff2]} developed this technique and independently derived heat
kernel estimates on Riemannian manifolds satisfying volume doubling
property and the Poincar\'e inequality. There has been lots of work
on improving heat kernel estimates on Riemannian manifolds, and
generalizing heat kernel estimates to general spaces, see excellent
surveys \cite{[Grig2],[Grig3],[Saloff2]} and references therein.

In this paper we will investigate heat kernel estimates on smooth
metric measure spaces and various applications. Let $(M,g)$ be an
$n$-dimensional complete Riemannian manifold, and let $f$ be a
smooth function on $M$. Then the triple $(M,g,e^{-f}dv)$ is called a
complete smooth metric measure space, where $dv$ is the volume
element of $g$, and $e^{-f}dv$ (for short, $d\mu$) is called the
weighted volume element or the weighted measure. On a smooth metric
measure space, the $m$-Bakry-\'Emery Ricci curvature
\cite{[BE],[Qian],[Lott1]} is defined by
\[
\mathrm{Ric}_f^m:=\mathrm{Ric}+\nabla^2f-\frac{1}{m}df\otimes df,
\]
where $\mathrm{Ric}$ is the Ricci curvature of $(M, g)$, $\nabla^2$
is the Hessian with respect to $g$, and
$m\in\mathbb{R}\cup\{\pm\infty\}$ (when $m=0$ we require $f$ to be a
constant). $m$-Bakry-\'Emery Ricci curvature is a natural
generalization of Ricci curvature on Riemannian manifolds, see
\cite{[BE],[BQ1],[Lott1],[Lo-Vi],[WW]} and references therein. In
particular, a smooth metric measure space satisfying
\[
\mathrm{Ric}_f^m=\lambda g,
\]
for some $\lambda\in\mathbb{R}$, is called an $m$-quasi-Einstein
manifold (see \cite{[CSW]}), which can be considered as natural
generalization of Einstein manifold. When $0<m<\infty$, $(M^n\times
F^m, g_M+e^{-2\frac{f}{m}}g_F)$, with $(F^m,g_F)$ an Einstein
manifold, is a warped product Einstein manifold. When $m=2-n$,
$(M^n, g)$ is a conformally Einstein manifold, in fact $\bar
g=e^{\frac{f}{(n-2)}}g$ is the Einstein metric. When $m=1$, $(M^n,
g)$ is the so-called static manifold in general relativity. When
$m=\infty$, we write
\[\mathrm{Ric}_f=\mathrm{Ric}_f^{\infty},\]
and the quasi-Einstein equation reduces to a gradient Ricci soliton.
The gradient Ricci soliton is called shrinking, steady, or
expanding, if $\lambda>0$, $\lambda=0$, or $\lambda<0$,
respectively. Ricci solitons play an important role in the Ricci
flow and Perelman's resolution of Poincar\'e conjecture and
geometrization conjecture, see \cite{[Cao1],[Hami]} and references
therein for nice surveys.

On a smooth metric measure space $(M,g,e^{-f}dv)$, the $f$-Laplacian
$\Delta_f$ is defined as
\[
\Delta_f=\Delta-\nabla f\cdot\nabla,
\]
which is self-adjoint with respect to $e^{-f}dv$. The $f$-heat
equation is defined as
\[
(\partial_t-\Delta_f)u=0.
\]
We denote the $f$-heat kernel by $H(x,y,t)$, that is, for each $y\in
M$, $H(x,y,t)=u(x,t)$ is the minimal positive solution of the
$f$-heat equation satisfying the initial condition $\lim_{t\to
0}u(x,t)=\delta_{f,y}(x)$, where $\delta_{f,y}(x)$ is the $f$-delta
function defined by
\[
\int_M\phi(x)\delta_{f,y}(x)e^{-f}dv=\phi(y)
\]
for any $\phi\in C_0^{\infty}(M)$. Similarly a function $u$ is said
to be $f$-harmonic if $\Delta_f u=0$, and $f$-subharmonic
($f$-superharmonic) if $\Delta_fu\geq 0\ (\Delta_f u\leq 0)$. It is
easy to see that the absolute value of an $f$-harmonic function is a
nonnegative $f$-subharmonic function. The weighted $L^p$-norm (or
$L_f^p$-norm) is defined as
\[
\|u\|_p=\left(\int_M |u|^pe^{-f}dv\right)^{1/p}
\]
for any $0<p<\infty$. We say that $u$ is $L_f^p$-integrable, i.e.
$u\in L_f^p$, if $\|u\|_p<\infty$.

\

Recall that for Riemannian manifolds,
using the classical Bochner formula, Li-Yau \cite{[Li-Yau1]} derived
the gradient estimate and heat kernel estimate. For smooth metric
measure spaces with $m<\infty$, there is an analogue of Bochner
formula for $\mathrm{Ric}_f^m$,
\begin{equation} \label{weightedBochner}
\begin{split}
\frac{1}{2}\Delta_f|\nabla u|^2=&|\nabla^2
u|^2+\langle\nabla\Delta_f u, \nabla u\rangle+\mathrm{Ric}_f^m(\nabla u, \nabla u)+\frac{1}{m}|\langle\nabla f,\nabla u\rangle|^2\\
\geq& \frac{(\Delta_f u)^2}{m+n}+\langle\nabla\Delta_f u, \nabla
u\rangle+\mathrm{Ric}_f^m(\nabla u, \nabla u).
\end{split}
\end{equation}
Therefore when $m<\infty$, the Bochner formula for
$\mathrm{Ric}_f^m$ can be considered as the Bochner formula for the
Ricci tensor of an $(n+m)$-dimensional manifold, and for smooth
metric measure spaces with $\mathrm{Ric}_f^m$ bounded below,
one has nice $f$-mean curvature comparison and $f$-volume comparison
theorems which are similar to classical ones for Riemannian
manifolds, see \cite{[BQ1], [WW]}, in particular, the comparison
theorems do not depend on $f$; X.-D. Li \cite{[LD]} derived an
analogue of Li-Yau gradient estimate, using which he proved $f$-heat
kernel estimates and several Liouville theorems; and in
\cite{[ChLu]}, by analyzing a family of warped product manifolds,
Charalambous and Z. Lu obtained $f$-heat kernel estimates and
essential spectrum.

Unfortunately when $m=\infty$, due to the lack of the extra term
$\frac{1}{m}|\langle\nabla f,\nabla u\rangle|^2$ in the Bochner
formula \eqref{weightedBochner}, one can only derive local $f$-mean
curvature comparison and local $f$-volume comparison (see
\cite{[WW]}) which highly rely on the potential function $f$, and
this makes it much more difficult to investigate smooth metric
measure spaces with $\mathrm{Ric}_f$ bounded below. According to
\cite{[MuWa], [MuWa2]}, there seems essential obstacles to derive
Li-Yau gradient estimate directly using the Bochner formua
\eqref{weightedBochner}, even with strong growth assumption on $f$.
It is interesting to point out that for $f$-harmonic functions,
Munteanu and Wang \cite{[MuWa], [MuWa2]} obtained Yau's gradient
estimate using both Yau's idea and the De Giorgi-Nash-Moser theory,
under appropriate assumptions on $f$.

In this paper we observe that, without any assumption on $f$, one
can derive a Harnack inequality for positive solutions of $f$-heat
equation, and local Gaussian bounds for the $f$-heat kernel on
smooth metric measure spaces using the De Giorgi-Nash-Moser theory.
Moreover, similar to \cite{[MuWa], [MuWa2]}, in each step one needs to figure
out the accurate coefficients, which play
key roles in the applications. As applications, we prove a Liouville
theorem for $f$-subharmonic functions, eigenvalues estimates for the
$f$-Laplacian, and $f$-Green's functions estimates.

Let us first state the local $f$-heat kernel estimates,

\begin{theorem}\label{Main1}
Let $(M,g,e^{-f}dv)$ be an $n$-dimensional complete noncompact
smooth metric measure space with $\mathrm{Ric}_f\geq -(n-1)K$ for
some constant $K\geq0$. For any point $o\in M$ and $R>0$, denote
\[
A(R)=\sup_{x\in B_o(3R)}|f(x)|,\quad A'(R)=\sup_{x\in
B_o(3R)}|\nabla f(x)|.
\]
Then for any $\epsilon>0$, there exist constants $c_1(n,\epsilon)$,
$c_i(n)$, $2\leq i\leq 6$  such that
\begin{equation}\label{uppfu1}
\begin{split}
&\frac{c_1\,e^{c_2A+c_3(1+A)\sqrt{Kt}}}
{V_f(B_x(\sqrt{t})^{1/2}V_f(B_y(\sqrt{t})^{1/2}}
\exp\left(-\frac{d^2(x,y)}{(4+\epsilon)t}\right)\\
&\hspace{4.8cm}\geq H(x,y,t)\geq
\frac{c_4e^{-c_5(A'^2+K)t}}{V_f(B_x(\sqrt{t}))}
\exp\left(-\frac{d^2(x,y)}{c_6t}\right)
\end{split}
\end{equation}
for all $x,y\in B_o(\frac 12R)$ and $0<t<R^2/4$. $\lim_{\epsilon\to
0}c_1(n,\epsilon)=\infty$.
\end{theorem}

When $f$ is bounded, the first author \cite{[Wu2]} obtained $f$-heat
kernel upper and lower bounds estimates. When $\mathrm{Ric}_f\geq 0$, the
authors \cite{[WuWu]} obtained $f$-heat kernel upper bound estimates
without assumptions on $f$.

It is worthy to point out that the lower bound estimate is sharp.
Indeed, let $(\mathbb{R},\ g_0, e^{-f}dx)$ be a $1$-dimensional
steady Gaussian soliton, where $g_0$ is the Euclidean metric and
$f(x)=\pm x$. From \cite{[WuWu]} the $f$-heat kernel is given by
\[
H(x,y,t)=\frac{e^{\pm \frac{x+y}{2}}\cdot e^{-t/4}}{(4\pi t)^{1/2}}
\times\exp\left(-\frac{|x-y|^2}{4t}\right).
\]
Obviously, the lower bound estimate is achieved by the above
$f$-heat kernel for steady Gaussian soliton as long as $t$ is very
large.

\begin{remark}
The factor $A'$ in the lower bound estimate comes from Harnack
inequality in Theorem \ref{Hartheorem}, it will be more interesting
to derive a new sharp lower bound in terms of $A$ instead of $A'$, if
possible.
\end{remark}

The proof of upper bound estimate of the $f$-heat kernel uses a
weighted mean value inequality and Davies's integral estimate
\cite{[Davies]}. The proof of lower bound estimate follows from a
Harnack inequality and a chaining argument, while the proof of the
Harnack inequality, follows from the arguments in \cite{[Saloff1],
[Saloff2]}.

To state the Harnack inequality, let us first introduce notations,
for any point $x\in M$ and $r>0$, $s\in\mathbb{R}$, and
$0<\varepsilon<\eta<\delta<1$, we denote $B=B_x(r)$, $\delta
B=B_x(\delta r)$ and
\[
Q=B\times (s-r^2,s),\quad Q_-=\delta B\times (s-\delta r^2,s-\eta
r^2), \quad Q_+=\delta B\times (s-\varepsilon r^2,s).
\]

\begin{theorem}\label{Hartheorem}
Let $(M,g,e^{-f}dv)$ be an $n$-dimensional complete noncompact
smooth metric measure space with $\mathrm{Ric}_f\geq-(n-1)K$ for
some constant $K\geq0$. Let $u$ be a positive solution to the
$f$-heat equation in $Q$, there exist constants $c_1$ and $c_2$
depending on $n$, $\varepsilon$, $\eta$ and $\delta$, such that
\[
\sup_{Q_-}\ u\leq c_1e^{c_2(A'^2+K)r^2}\inf_{Q_+}\ u,
\]
where $A'(r)=\sup_{y\in B_x(3r)}|\nabla f(y)|$.
\end{theorem}

By a different volume comparison, we get another form of Harnack
inequality and lower bound estimate for the $f$-heat kernel,
\begin{theorem} \label{lowerbound2}
Under the assumptions of Theorem \ref{Hartheorem} and Theorem
\ref{Main1}, respectively, we have
\[
\sup_{Q_-}\{u\}\leq
\exp\left\{c_1e^{c_2A}\left[(1+A^2)Kr^2+1\right]\right\}\cdot\inf_{Q_+}\{u\},
\]
where $A=A(r)=\sup_{y\in B_x(3r)}|f(y)|$, and
\begin{equation}\label{lowhe2x}
H(x,y,t)\geq \frac{c_4}{V_f(B_x(\sqrt{t}))}
\times\exp\left[-c_5e^{c_6A}\left((1+A^2)Kt+1+\frac{d^2(x,y)}{t}\right)\right],
\end{equation}
for all $x,y\in B_o(\frac 12R)$ and $0<t<R^2/4$, where
$A=A(R)=\sup_{x\in B_o(3R)}|f(x)|$. In particular, when $f$ is
bounded, we get
\begin{equation}\label{lowhe2x}
H(x,y,t)\geq \frac{c_1e^{-c_2Kt}}{V_f(B_x(\sqrt{t}))}
\times\exp\left(-\frac{d^2(x,y)}{c_3t}\right).
\end{equation}
\end{theorem}

\

Next we derive several applications of the $f$-heat kernel
estimates. First we prove a Liouville theorem for $f$-subharmonic
functions. Recall Pigola, Rimoldi, Setti \cite{[PRG]} proved that
any nonnegative $L_f^1$-integrable $f$-superharmonic function must
be constant if $\mathrm{Ric}_f$ is bounded below, without any
assumption on $f$. However, as proved in \cite{[WuWu]}, for
$f$-subharmonic functions, the condition on $f$ is necessary. In
fact we provided explicit counterexamples illustrating that $f$
cannot grow faster than quadratically when $\mathrm{Ric}_f\geq 0$.
Below we show that the $L_f^1$-Liouville theorem also holds for
$f$-subharmonic functions when $\mathrm{Ric}_f\geq -(n-1)K$ and $f$
has at most linear growth.

\begin{theorem}\label{Mainthm2}
Let $(M,g,e^{-f}dv)$ be an $n$-dimensional complete noncompact
smooth metric measure space with $\mathrm{Ric}_f\geq -(n-1)K$ for
some constant $K>0$. Assume there exist nonnegative constants $a$
and $b$ such that
\[
|f|(x)\leq ar(x)+b,
\]
where $r(x)$ is the distance function to a fixed point $o\in M$.
Then any nonnegative $L_f^1$-integrable $f$-subharmonic function
must be identically constant. In particular, any $L_f^1$-integrable
$f$-harmonic function must be identically constant.
\end{theorem}

There have been various Liouville type theorems for $f$-subharmonic
and $f$-harmonic functions on smooth metric measure spaces and
gradient Ricci solitons under different conditions, see Brighton
\cite{[Bri]}, Cao-Zhou \cite{[CaoZhou]}, Munteanu-Sesum
\cite{[MuSe]}, Munteanu-Wang \cite{[MuWa],[MuWa2]}, Petersen-Wylie
\cite{[PW]}, and Wei-Wylie \cite{[WW]} for details.

By a similar argument in \cite{[Li0]} (see also \cite{[WuWu]}), we
also prove an $L_f^1$-uniqueness theorem for solutions of $f$-heat
equation, see Theorem \ref{Main3} in Section \ref{sec4}.

\

Second we derive lower bound estimates for eigenvalues of the
$f$-Laplace operator on compact smooth metric measure spaces, by
adapting the classical argument of Li-Yau \cite{[Li-Yau1]},

\begin{theorem}\label{eigenva1}
Let $(M,g,e^{-f}dv)$ be an $n$-dimensional compact smooth metric
measure space with $\mathrm{Ric}_f\geq -(n-1)K$ for some constant
$K\geq0$. Let $0=\lambda_0<\lambda_1\leq\lambda_2\leq\ldots$ be
eigenvalues of the $f$-Laplacian $\Delta_f$. Then there exists a
constant $C$ depending only on $n$ and $A=\max_{x\in M}f(x)$, such
that
\begin{alignat*}{2}
\lambda_k\geq&\frac{C(k+1)^{2/n}}{d^2},&&\quad K=0,\\
\lambda_k\geq&\frac{C}{d^2}\left(\frac{k+1}{\exp(C\sqrt{K}d)}\right)^{\frac{2}{n+4A}},&&\quad
K>0,
\end{alignat*}
for all $k\geq1$, where $d$ is the diameter of $M$.
\end{theorem}

The upper bound estimates were proved by
Hassannezhad \cite{[Hass]}, and Colbois, Soufi, Savo \cite{[CSS]},
which depend on norms of the potential function and conformal class of the
metric. For the first eigenvalue, there have been more interesting
results. When $M$ is compact and $\mathrm{Ric}_f\geq\frac{a}{2}>0$,
Andrews, Ni \cite{[AN]}, and Futaki, Li, Li \cite{[FLL]} derived
lower bound estimates for the first eigenvalue, which depend on the
diameter of the manifolds. When $M$ is complete noncompact,
Munteanu, Wang \cite{[MuWa],[MuWa2],[MuWa3]}, and Wu \cite{[Wu3]}
obtained first eigenvalues estimates under appropriate assumptions
on $f$. Cheng, Zhou \cite{[CZ]} proved an interesting Obata type
theorem.\\

At last we discuss $f$-Green's functions estimates. We first get
upper and lower estimates for $f$-Green's functions when $f$ is
bounded, which is similar to the classical estimates of Li-Yau
\cite{[Li-Yau1]} for Riemannian manifolds.
Recall the $f$-Green's function on $(M,g,e^{-f}dv)$ is defined as
\[
G(x,y)=\int^\infty_0H(x,y,t)dt
\]
if the integral on the right hand side converges. It is easy to
check that $G$ is positive and satisfies
\[
\Delta_f G=-\delta_{f,y}(x).
\]

\begin{theorem}\label{ThmGreen1}
Let $(M,g,e^{-f}dv)$ be an $n$-dimensional complete noncompact
smooth metric measure space with $\mathrm{Ric}_f\geq 0$ and $f$
bounded. If $G(x,y)$ exists, then there exist constants $c_1$ and
$c_2$ depending only on $n$ and $\sup f$, such that
\begin{equation}\label{Greenest}
c_1\int^\infty_{r^2}V^{-1}_f(B_x(\sqrt{t}))dt\leq G(x,y)\leq
c_2\int^\infty_{r^2}V^{-1}_f(B_x(\sqrt{t}))dt,
\end{equation}
where $r=r(x,y)$.
\end{theorem}

Recently Dai, Sung, Wang, and Wei \cite{[DSWW]} observed that  every
gradient steady Ricci soliton admits a positive $f$-Green's
function, hence it is $f$-nonparabolic. We provide an alternative
proof using a criterion of Li-Tam \cite{[LiTa1],[LiTa2]}, and the
$f$-heat kernel for steady Gaussian Ricci soliton,

\begin{theorem}\label{posGreen} Let $(M^n,g,f)$ be
a complete gradient steady soliton. Then there exists a positive
smooth $f$-Green function, therefore the gradient steady soliton is
$f$-nonparabolic.
\end{theorem}

In \cite{[SWW]}, Song, Wei and Wu investigated several properties of
$f$-Green's functions on smooth metric measure spaces. Pigola, Rimoldi,
and Setti \cite{[PRG]} proved that gradient shrinking Ricci solitons are $f$-parabolic.\\

The paper is organized as follows. In Section \ref{sec2a}, we recall
comparison theorems for the Bakry-\'Emery Ricci curvature bounded
below, using which we derive a local $f$-volume doubling property, a
local $f$-Neumann Poincar\'e inequality, a local Sobolev inequality
and mean value inequalities for the $f$-heat equation. In Section
\ref{sec3v}, we prove a Moser's Harnack inequality of $f$-heat
equation following the arguments of Saloff-Coste in \cite{[Saloff1],
[Saloff2]}. In Section \ref{sec3}, we  prove local Gaussian upper
and lower bound estimates of the $f$-heat kernel. In Section
\ref{sec4}, following the same arguments of \cite{[WuWu]}, we
establish a new $L_f^1$-Liouville theorem for the $f$-harmonic
function and a new $L_f^1$-uniqueness property for nonnegative
solutions of the $f$-heat equation. In Section \ref{eigen}, we apply
upper bounds of the $f$-heat kernel to get the eigenvalue estimates
of the $f$-Laplacian on compact smooth metric measure spaces. In
Section \ref{Greenf}, we derive Green function estimates for smooth
metric measure spaces with $\mathrm{Ric}_f\geq 0$ and $f$ bounded,
and for gradient steady Ricci solitons.

\textbf{Acknowledgement}. The authors thank Professors Xiaodong Cao
and Laurent Saloff-Coste for their suggestions and great help. The
second author thanks Professors Xianzhe Dai and Guofang Wei for
helpful discussions, guidance, constant encouragement and support.
The first author is partially supported by NSFC (11101267, 11271132)
and the China Scholarship Council (201208310431). The second author
is partially supported by an AMS-Simons travel grant.

\section{Poincar\'e, Sobolev and mean value inequalities}
\label{sec2a}

Recall that for any point $p\in M$ and $R>0$, we denote
\[
A(R)=A(p,R)=\sup_{x\in B_p(3R)}|f(x)|,\quad A'(R)=A'(p,R)=\sup_{x\in
B_p(3R)}|\nabla f(x)|.
\]
When there is no confusion we write $A$, $A'$ for short. We start
from the relative $f$-volume comparison theorem of Wei and Wylie
\cite{[WW]}.
\begin{lemma}\label{comp}
Let $(M,g,e^{-f}dv)$ be an $n$-dimensional complete noncompact
smooth metric measure space. If $\mathrm{Ric}_f\geq-(n-1)K$ for some
constant $K\geq0$, then
\begin{equation}\label{volcomp}
\frac{V_f(B_x(R_1,R_2))}{V_f(B_x(r_1,r_2))}\leq
\frac{V^{n+4A}_K(B_x(R_1,R_2))}{V^{n+4A}_K(B_x(r_1,r_2))}
\end{equation}
for any $0<r_1<r_2,\ 0<R_1<R_2$, $r_1\leq R_1,\ r_2\leq R_2$, where
$B_x(R_1,R_2)=B_x(R_2)\backslash B_x(R_1)$, and
$A=A(x,\frac{1}{3}R_2)$. Here ${V^{n+4A}_K(B_x(r))}$ denotes the
volume of the ball in the model space $M^{n+4A}_K$, i.e., the simply
connected space form with constant sectional curvature $-K$ and
dimension $n+4A$.

Similarly we have
\begin{equation}\label{volcomp2}
\frac{V_f(B_x(R_1,R_2))}{V_f(B_x(r_1,r_2))}\leq
\frac{V^{n+4A'R_2}_K(B_x(R_1,R_2))}{V^{n+4A'R_2}_K(B_x(r_1,r_2))},
\end{equation}
where $A'=A'(x,\frac{1}{3}R_2)$.
\end{lemma}
\begin{remark} \label{replace}
Following the proofs, $A(R)$ in all following lemmas, propositions,
theorems and corollaries can be replaced by $RA'(R)$. We will apply
the first volume comparison \eqref{volcomp} to derive heat kernel
upper bound, and the second volume comparison \eqref{volcomp2} to
derive Harnack inequality and heat kernel lower bound.
\end{remark}

\begin{proof}[Proof of Lemma \ref{comp}]
Applying the weighted Bochner formula (\ref{weightedBochner}) and an
ODE argument, Wei and Wylie (see (3.19) in \cite{[WW]}) proved the
following $f$-mean curvature comparison theorem. Recall that the
weighted mean curvature $m_f(r)$ is defined as
\[
m_f(r)=m(r)-\nabla f\cdot\nabla r=\Delta_f\ r.
\]
If $Ric_f\geq-(n-1)K$, then
\begin{equation}
\begin{split}
m_f(r)&\leq(n-1)\sqrt{K}\coth(\sqrt{K}\,r)+\frac{2K}{\sinh^2(\sqrt{K}\,r)}\int^r_0(f(t)-f(r))\cosh(2\sqrt{K}\,t)dt\\
&\leq(n-1+4A)\sqrt{K}\cdot\coth(\sqrt{K}\,r)
\end{split}
\end{equation}
along any minimal geodesic segment from $x$. In geodesic polar
coordinates, the volume element is written as
\[
dv=\mathcal{A}(r,\theta)dr\wedge d\theta_{n-1},
\]
where $d\theta_{n-1}$ is the standard volume element of the unit
sphere $S^{n-1}$. Let
\[
\mathcal{A}_f(r,\theta)=e^{-f}\mathcal{A}(r,\theta).
\]
By the first variation of the area,
\[
\frac{\mathcal{A'}}{\mathcal{A}}(r,\theta)=(\ln(\mathcal{A}(r,\theta)))'=m(r,\theta).
\]
Therefore
\[
\frac{\mathcal{A'}_f}{\mathcal{A}_f}(r,\theta)=(\ln(\mathcal{A}_f(r,\theta)))'=m_f(r,\theta),
\]

So for $r<R$,
\[
\frac{\mathcal{A}_f(R,\theta)}{\mathcal{A}_f(r,\theta)}\leq
\frac{\mathcal{A}^{n+4A}_K(R)}{\mathcal{A}^{n+4A}_K(r)},
\]
That is $\frac{\mathcal{A}_f(r,\theta)}{\mathcal{A}^{n+4A}_K(r)}$ is
nonincreasing in $r$, where $\mathcal{A}^{n+4A}_K(r)$ is the volume
element in the simply connected hyperbolic space of constant
sectional curvature $-K$ and dimension $n+4A$. Applying Lemma 3.2 in
\cite{[Zhu]}, we get
\[
\frac{\int^{R_2}_{R_1}\mathcal{A}_f(R,\theta)dt}{\int^{r_2}_{r_1}\mathcal{A}_f(r,\theta)dt}\leq
\frac{\int^{R_2}_{R_1}\mathcal{A}^{n+4A}_K(R,\theta)dt}{\int^{r_2}_{r_1}\mathcal{A}^{n+4A}_K(r,\theta)dt}
\]
for $0<r_1<r_2$, $0<R_1<R_2$, $r_1\leq R_1$ and $r_2\leq R_2$.
Integrating along the sphere direction proves \eqref{volcomp}.

The second volume comparison (\ref{volcomp2}) follows from an
observation for the weighted mean curvature,
\begin{equation}
\begin{split}
m_f(r)&\leq(n-1)\sqrt{K}\coth(\sqrt{K}\,r)+\frac{2K}{\sinh^2(\sqrt{K}\,r)}\int^r_0(f(t)-f(r))\cosh(2\sqrt{K}\,t)dt\\
&\leq(n-1+4A'r)\sqrt{K}\cdot\coth(\sqrt{K}\,r).
\end{split}
\end{equation}
\end{proof}

Let $V^{n+4A}_K(B_x(r))$ be the volume of the ball of radius $r$ in
the simply connected hyperbolic space of constant sectional
curvature $-K$ and dimension $n+4A$. If $K>0$, the model space is
the hyperbolic space. If $K=0$, the model space is the Euclidean
space. In any case, we have the estimate
\begin{equation}\label{props1}
\omega_{n+4A}\cdot r^{n+4A}\leq V_K(B_x(r))\leq \omega_{n+4A}\cdot
r^{n+4A}e^{(n-1+4A)\sqrt{K}\,r}
\end{equation}
where $\omega_{n+4A}$ is the volume of the unit
ball in $(n+4A)$-dimensional Euclidean space.\\

Similar to \cite{[WuWu]}, Lemma \ref{comp} implies a local
$f$-volume doubling property. Indeed, in \eqref{volcomp}, letting
$r_1=R_1=0$, $r_2=r$ and $R_2=2r$, from \eqref{props1} we get
\begin{equation}
\begin{aligned}\label{voldop}
V_f(B_x(2r))\leq 2^{n+4A}e^{2(n-1+4A)\sqrt{K}\,r}\cdot V_f(B_x(r))
\end{aligned}
\end{equation}
This local $f$-volume doubling property is crucial in our proof of
Poincar\'e inequality, Sobolev inequality, mean-value inequality, and Harnack inequality.\\

From Lemma \ref{comp}, we also have the following,
\begin{lemma}\label{lecomp1}
Let $(M,g,e^{-f}dv)$ be an $n$-dimensional complete noncompact
smooth metric measure space. If $\mathrm{Ric}_f\geq-(n-1)K$ for some
constant $K>0$, then
\[
V_f(B_x(r))\leq
\frac{e^{(n-1+4A)\sqrt{K}(d(x,y)+r)}}{r^{n+4A}}V_f(B_y(r)),
\]
where $A=A(y,d(x,y)+r)$.
\end{lemma}
\begin{proof}
We let $r_1=0$, $r_2=r$, $R_1=d(x,y)-r$ and $R_2=d(x,y)+r$ in Lemma
\ref{comp}. Then using \eqref{props1} we have
\[
\frac{V_f(B_y(d(x,y)+r))-V_f(B_y(d(x,y)-r))}{V_f(B_y(r))}
\leq\frac{e^{(n-1+4A)\sqrt{K}(d(x,y)+r)}}{r^{n+4A}}.
\]
Therefore we get
\begin{equation*}
\begin{aligned}
V_f(B_x(r))&\leq V_f(B_y(d(x,y)+r))-V_f(B_y(d(x,y)-r))\\
&\leq \frac{e^{(n-1+4A)\sqrt{K}(d(x,y)+r)}}{r^{n+4A}}V_f(B_y(r)).
\end{aligned}
\end{equation*}
\end{proof}

Following the argument of \cite{[Bus]} (see also \cite{[Saloff2]} or
\cite{[MuWa]}), applying Lemma \ref{comp} we get a local Neumann
Poincar\'e inequality on complete smooth metric measure spaces.
\begin{lemma}\label{NeuPoin}
Let $(M,g,e^{-f}dv)$ be an $n$-dimensional complete noncompact
smooth metric measure space with $\mathrm{Ric}_f\geq-(n-1)K$ for
some constant $K\geq0$. Then,
\begin{equation}\label{Nepoinineq}
\int_{B_x(r)}|\varphi-\varphi_{B_x(r)}|^2d\mu\leq
c_1e^{c_2A+c_3(1+A)\sqrt{K}r}\cdot r^2\int_{B_x(r)}|\nabla
\varphi|^2d\mu
\end{equation}
for any $\varphi\in C^\infty(B_x(r))$, where
$\varphi_{B_x(r)}=\int_{B_x(r)}\varphi d\mu/\int_{B_x(r)}d\mu$.
\end{lemma}
\begin{remark}
By Remark \ref{replace}, the coefficient $c_2A+c_3(1+A)\sqrt{K}r$ in
Lemma \ref{NeuPoin} and all following lemmas, propositions,
theorems, and corollaries, can be replaced by
$c_2(A'+\sqrt{K})r+c_3A'\sqrt{K}r^2$.
\end{remark}

Combining Lemma \ref{comp} and Lemma \ref{NeuPoin} and the argument
of \cite{[HK]} (see also \cite{[WuWu]}), we obtain a local Sobolev
inequality on smooth metric measure spaces.
\begin{lemma}\label{NeuSob}
Let $(M,g,e^{-f}dv)$ be an $n$-dimensional complete noncompact
smooth metric measure space with $\mathrm{Ric}_f\geq-(n-1)K$ for
some constant $K\geq0$. Then there exists $\nu>2$, such that
\begin{equation}\label{loSob}
\left(\int_{B_x(r)}|\varphi|^{\frac{2\nu}{\nu-2}}d\mu
\right)^{\frac{\nu-2}{\nu}}\leq
\frac{c_1e^{c_2A+c_3(1+A)\sqrt{K}r}\cdot r^2}{V_f(B_x(r))^{\frac
2\nu}} \int_{B_x(r)}(|\nabla \varphi|^2+r^{-2}|\varphi|^2)d\mu
\end{equation}
for any $\varphi\in C^\infty(B_x(r))$.
\end{lemma}

Applying Lemma \ref{NeuSob} we obtain a mean value inequality for
solutions to the $f$-heat equation, which is similar to Theorem
5.2.9 in \cite{[Saloff2]} (see also \cite{[WuWu]}).
\begin{proposition}\label{PoinDouHarn}
Let $(M,g,e^{-f}dv)$ be an $n$-dimensional complete noncompact
smooth metric measure space. Assume \eqref{loSob} holds. Fix
$0<p<\infty$. There exist constants $c_1(n,p,\nu)$, $c_2(n,p,\nu)$
and $c_3(n,p,\nu)$ such that for any $s\in\mathbb{R}$ and
$0<\delta<1$, any smooth positive subsolution $u$ of the $f$-heat
equation in the cylinder $Q=B_x(r)\times(s-r^2,s)$ satisfies
\[
\sup_{Q_\delta}\{u^p\}\leq
\frac{c_1e^{c_2A+c_3(1+A)\sqrt{K}r}}{(1-\delta)^{2+\nu}\,
r^2\,V_f(B_x(r))}\cdot\int_Qu^p\,\,\, d\mu\, dt,
\]
where $Q_\delta=B_x(\delta r)\times (s-\delta r^2,s)$.
\end{proposition}

Similar to Proposition \ref{PoinDouHarn}, we have
\begin{proposition}\label{PoinDou2}
Let $(M,g,e^{-f}dv)$ be an $n$-dimensional complete noncompact
smooth metric measure space. Assume \eqref{loSob} holds. Fix
$0<p_0<1+\nu/2$. There exist constants $c_1(n,p_0,\nu)$,
$c_2(n,p_0,\nu)$ and $c_3(n,p_0,\nu)$ such that for any
$s\in\mathbb{R}$, $0<\delta<1$, and $0<p\leq p_0$, any smooth
positive supersolution $u$ of the $f$-heat equation in the cylinder
$Q=B_x(r)\times(s-r^2,s)$ satisfies
\[
{\|u\|^p}_{p_0,{Q^{'}_{\delta}}}\leq
\left\{\frac{c_1e^{c_2A+c_3(1+A)\sqrt{K}r}}{(1-\delta)^{2+\nu}\,
r^2\,V_f(B_x(r))}\right\}^{1-p/p_0} \cdot{\|u\|^p}_{p,Q},
\]
where ${Q'}_\delta:=B_x(\delta r)\times (s-r^2,s-(1-\delta)r^2)$. On
the other hand, for any $0<p<\bar{p}<\infty$, there exist constants
$c_4(n,\bar{p},\nu)$, $c_5(n,\bar{p},\nu)$ and $c_6(n,\bar{p},\nu)$
such that
\[
\sup_{Q_\delta}\{u^{-p}\}\leq
\frac{c_4e^{c_5A+c_6(1+A)\sqrt{K}r}}{(1-\delta)^{2+\nu}\,
r^2\,V_f(B_x(r))}\cdot{\|u^{-1}\|^p}_{p,Q}\,\,\, ,
\]
where $||u||_{p,Q}=\left(\int_Q|u(x,t)|^pd\mu dt\right)^{1/p}.$
\end{proposition}

\begin{proof}[Proof of Proposition \ref{PoinDou2}]
For any nonnegative test function $\phi\in C_0^{\infty}(B)$ and any
supersolution of the heat equation, we have
\[
\int_B (\phi\partial_t u+\nabla\phi\nabla u)d\mu\geq0.
\]
Let $\phi=\epsilon qu^{q-1}\psi^2$, $w=u^{q/2}$ for $-\infty<q\leq
p(1+\nu/2)^{-1}<1$ and $q\neq 0$, where $\epsilon=1$ if $q>0$ and
$\epsilon=-1$ if $q<0$. We get
\[
\epsilon\int_B \left(\psi^2\partial_t w^2+4(1-1/q)\psi^2|\nabla
w|^2+4w\psi\langle\nabla w,\nabla \psi\rangle\right)d\mu\geq 0.
\]
When $q>0$. Since
\[
2w\psi\langle\nabla w,\nabla \psi\rangle\geq- a^{-2}\psi^2|\nabla
w|^2-a^2w^2|\nabla\psi|^2
\]
for any $a>0$, we get
\[
-\int_B \psi^2\partial_t (w^2) d\mu+c_1\int_B |\nabla (\psi w)|^2
d\mu\leq c_2\|\nabla\psi\|^2_{\infty}\int_{\text{supp}(\psi)}
w^2d\mu,
\]
where $c_1$ and $c_2$ depend only on $q$. Multiplying a nonnegative
smooth function $\lambda(t)$, we have
\[
-\partial_t\int_B \lambda^2\psi^2w^2 d\mu+c_1\lambda^2\int_B |\nabla
(\psi w)|^2 d\mu\leq c_3\lambda(\lambda\|\nabla\psi\|^2_{\infty}
+\|\psi\lambda'\|_{\infty})\int_B w^2d\mu.
\]
Choose $\psi$ and $\lambda$ such that
\begin{equation*}
\begin{split}
0\leq\psi\leq 1,\ \text{supp}\psi\subset\sigma B,\ &\psi=1\
\text{in}\
\sigma' B,\ |\nabla\psi|\leq(\kappa r)^{-1},\\
0\leq\lambda\leq 1,\ \lambda=1\ \text{in}\ (-\infty,s-\sigma r^2],\
& \lambda=0\ \text{in}\ [s-\sigma' r^2, \infty),\ |\lambda'|\leq
(\kappa r^2)^{-1},
\end{split}
\end{equation*}
where $0<\sigma'<\sigma<1$, $\kappa=\sigma-\sigma'$. Let
$I_{\sigma}=[s-\sigma r^2, s]$, and integrate the above inequality
on $[s-r^2, t]$ for $t\in I_{\sigma'}$. We get
\begin{equation*}
\begin{split}
&\sup_{I_{\sigma'}}\int_{\sigma'B} w^2 d\mu +c_1\int_{Q_{\sigma'}}
|\nabla w|^2 d\mu dt\leq c_4(\kappa r)^{-2}\int_{Q_{\sigma}} w^2d\mu
dt.
\end{split}
\end{equation*}
By H\"older inequality and Proposition \ref{NeuSob}, for any
$\phi\in C_0^{\infty}(B)$, we get
\begin{equation*}
\begin{split}
\int_B \phi^{2(1+2/\nu)}d\mu&\leq \left(\int_B
\phi^{2\nu/(\nu-2)}d\mu\right)^{(\nu-2)/\nu}\left(\int_B \phi^2d\mu\right)^{2/\nu}\\
&\leq
C(B)\left(\int_B(|\nabla\phi|^2+r^{-2}\phi^2)d\mu\right)\left(\int_B
\phi^2d\mu\right)^{2/\nu},
\end{split}
\end{equation*}
where $C(B):=c_1e^{c_2A+c_3(1+A)\sqrt{K}r}\,r^2V_f^{-2/\nu}$.
Therefore
\begin{equation} \label{witerationineq}
\begin{split}
\int_{Q_{\sigma'}} u^{q\theta} d\mu dt&\leq
c_3C(B)\left((r\kappa)^{-2}\int_{Q_{\sigma}} u^qd\mu
dt\right)^{\theta},
\end{split}
\end{equation}
where $\theta=1+2/\nu$. Let $p_i=p_0\theta^{-i}$, notice that by
H\"older inequality, for any $p_i<p=\eta p_i+(1-\eta)p_{i-1}\leq
p_{i-1}$ with $0\leq\eta<1$,
\begin{equation*}
\|u\|_p^p\leq \|u\|_{p_i}^{\eta
p_i}\|u\|_{p_{i-1}}^{(1-\eta)p_{i-1}},
\end{equation*}
so it suffices to prove the estimate for all $p_i$.

Fix $i$, and let $q_j=p_i\theta^j$, $1\leq j\leq i-1$, so
$0<q_j<p_0(1+\nu/2)^{-1}$. Let $\sigma_0=1$,
$\sigma_i=\sigma_{i-1}-\kappa_i$, where $\kappa_i=(1-\delta)2^{-i}$,
so $\sigma_i=1-(1-\delta)\sum_1^i 2^{-j}>\delta$. Plugging into
inequality (\ref{witerationineq}), we get
\[
\int_{Q'_{\sigma_j}} u^{q_0\theta^j}d\mu dt\leq
c_4^jC(B)\left((1-\delta)^{-2}r^{-2}\int_{Q'_{\sigma_{j-1}}}
u^{q_0\theta^{j-1}}d\mu dt\right)^{\theta},
\]
for $1\leq j\leq i$. Therefore
\[
\int_{Q'_{\sigma_i}} u^{p_0}d\mu dt\leq c_4^{\sum (i-j)\theta^{j+1}}
C(B)^{\sum\theta^j}[(1-\delta)r]^{-2\sum\theta^{j+1}}\left(\int_{Q}
u^{p_i}d\mu dt\right)^{\theta^i},
\]
where the summation is taken from $0$ to $i-1$. Therefore we obtain
\[
\left(\int_{Q'_{\sigma_i}} u^{p_0}d\mu dt\right)^{p_i/p_0}
\leq[c_5(1-\delta)^{-2-\nu}E(B)]^{1-p_i/p_0}\left(\int_{Q}
u^{p_i}d\mu dt\right),
\]
where $E(B)=C(B)^{\nu/2}r^{-2-\nu}$.

\

When $q<0$. We get
\[
\int_B (\psi^2\partial_t w^2+4(1-1/q)\psi^2|\nabla
w|^2+4w\psi\langle\nabla w,\nabla \psi\rangle) d\mu\leq 0.
\]
Applying the mean value inequality to the last term, we get
similarly
\[
\int_B \psi^2\partial_t (w^2) d\mu+c_6\int_B |\nabla (\psi w)|^2
e^{-f}dv\leq c_7\|\nabla\psi\|^2_{\infty}\int_{\text{supp}(\psi)}
w^2 d\mu.
\]
By the above argument, we can obtain
\[
\int_{Q_{\sigma'}} w^{2\theta} d\mu dt\leq
c_8C(B)\left((r\kappa)^{-2}\int_{Q_{\sigma}} w^2d\mu
dt\right)^{\theta},
\]
where $\theta=1+2/\nu$. For any $\alpha>1$, $v=u^{\alpha}$ satisfies
\[
\partial_t v-\Delta_f v\geq-\frac{\alpha-1}{\alpha}v^{-1}|\nabla v|^2,
\]
applying the above argument again, we also have
\[
\int_{Q_{\sigma'}} w^{2\alpha\theta }d\mu dt\leq
c_9C(B)\left((r\kappa)^{-2}\int_{Q_{\sigma}} w^{2\alpha}d\mu
dt\right)^{\theta}.
\]
Let $\kappa_i=(1-\delta)2^{-i-1}$, and $\sigma_0=1$,
$\sigma_i=\sigma_{i-1}-\kappa_i=1-\sum_1^i \kappa_j$, and
$\alpha_i=\theta^i$. We get
\begin{equation*}
\begin{split}
\left(\int_{Q_{\sigma_{i+1}}} w^{2\theta^{i+1} }d\mu
dt\right)^{\theta^{-i-1}}&\leq
C(B)\left(c_{10}^{i+1}[(1-\delta)r]^{-2}\int_{Q_{\sigma_{i}}}
w^{2\theta^i}d\mu dt\right)^{\theta}\\
&\leq C(B)^{\sum\theta^{-j-1}}c_{10}^{\sum
(j+1)\theta^{-j-1}}[(1-\delta)r]^{-2\sum\theta^{-j}}\int_Q w^2d\mu
dt,
\end{split}
\end{equation*}
where the summation is from $1$ to $i+1$. Therefore when
$i\rightarrow \infty$, we get
\[
\sup_{Q_{\delta}} w^2 \leq c_5
C(B)^{\nu/2}[(1-\delta)r]^{-2-\nu}\|w\|^2_{2,Q}
\]
and the conclusion follows.
\end{proof}


\section{Moser's Harnack inequality for $f$-heat equation}\label{sec3v}
In this section we prove Moser's Harnack inequalities for the
$f$-heat equation using Moser iteration, which will lead to the
sharp lower bound estimate for the $f$-heat kernel in the next
section. The arguments mainly follow those in
\cite{[Moser],[Moser2],[Saloff1],[Saloff2]}, while more delicate
analysis is required to get the accurate estimates, which depend on
the potential function. Throughout this section, we will use the
second $f$-volume comparison, i.e., \eqref{volcomp2} in Section 2.

Recall the notations defined in Introduction. for any point $x\in M$
and $r>0$, $s\in\mathbb{R}$, and $0<\varepsilon<\eta<\delta<1$, we
denote $B=B_x(r)$, $\delta B=B_x(\delta r)$ and
\[
Q=B\times (s-r^2,s),\quad Q_\delta=\delta B\times (s-\delta r^2,s),
\quad {Q'}_\delta=\delta B\times (s- r^2,s-(1-\delta)r^2),\]
\[Q_-=\delta B\times (s-\delta r^2,s-\eta
r^2), \quad Q_+=\delta B\times (s-\varepsilon r^2,s).\]

With the above notations, we have the main result in this section.
\begin{theorem}\label{Harthm}
Let $(M,g,e^{-f}dv)$ be an $n$-dimensional complete noncompact
smooth metric measure space with $\mathrm{Ric}_f\geq-(n-1)K$ for
some constant $K\geq0$. For any point $x\in M$, $r>0$, and any
parameters $0<\varepsilon<\eta<\delta<1$, let $u$ be a smooth
solution of the $f$-heat equation in $Q$, then there exist constants
$c_1$ and $c_2$ both depending on $n$, $\varepsilon$, $\eta$ and
$\delta$, such that
\[
\sup_{Q_-}\ u\leq c_1e^{c_2(A'^2+K)r^2}\inf_{Q_+}\ u,
\]
where $A'=A'(x,r+1)$.
\end{theorem}
\begin{remark}
The coefficient in Theorem \ref{Harthm} comes from the second volume
comparison Lemma \ref{comp}. On the other hand, the first volume
comparison in Lemma \ref{comp} leads to another Harnack inequality,
\begin{equation*}
\sup_{Q_-}\{u\}\leq
\exp\left\{c_1e^{c_2A}\left[(1+A^2)Kr^2+1\right]\right\}\cdot\inf_{Q_+}\{u\}.
\end{equation*}
Since its proof is very similar to that of Theorem \ref{Harthm}, we
omit the proof here.
\end{remark}

\vspace{1em}

We first modify the $f$-Poincar\'e inequality \eqref{Nepoinineq} in
Section 2 to a weighted version, which can be derived by adapting a
Whitney-type covering argument, see Sections 5.3.3-5.3.5 in
\cite{[Saloff2]},

Let $\xi:[0,\infty)\to [0,1]$ be a non-increasing function such that
$\xi(t)=0$ for $t>1$, and for some positive constant $\beta$
\[
\xi\left(t+\frac{1-t}{2}\right)\geq \beta\xi(t),\quad 1/2\leq t\leq
1.
\]
Let $\Psi_B(z):=\xi(\rho(x,z)/r)$ for $z\in B=B(x,r)$ and
$\Psi_B(z)=0$ for $z\in M\setminus B$, we write $\Psi(z)$ for short.
Then
\begin{lemma}\label{weigNP}
Let $(M,g,e^{-f}dv)$ be an $n$-dimensional complete noncompact
smooth metric measure space with $\mathrm{Ric}_f\geq-(n-1)K$ for
some constant $K\geq0$. There exist constants $c_1(n,\xi)$, $c_2(n)$
and $c_3(n)$ such that, for any $B_x(r)\subset M$, we have
\begin{equation}\label{weiNep}
\int_{B_x(r)}|\varphi-\varphi_{\Psi}|^2\Psi d\mu\leq
c_1e^{c_2(A'+\sqrt{K})r+c_3A'\sqrt{K}r^2}\cdot
r^2\int_{B_x(r)}|\nabla\varphi|^2\Psi d\mu
\end{equation}
for all $\varphi\in C^\infty(B_x(r))$, where
$\varphi_\Psi=\int_B\varphi\Psi d\mu/\int_B\Psi d\mu$.
\end{lemma}

\vspace{1em}

Secondly, for a positive solution $u$ to the $f$-heat equation, we
derive an estimate for the level set of $\log u$, the proof of which
depends on Lemma \ref{weigNP}. This inequality is important for the
iteration arguments in Lemma \ref{crilem2}. In the following, we
denote $d\bar{\mu}=d\mu\times dt$ by the natural product measure on
$M\times \mathbb{R}$.
\begin{lemma}\label{crilem1}
Let $(M,g,e^{-f}dv)$ be an $n$-dimensional complete noncompact
smooth metric measure space. Assume that \eqref{voldop} and
\eqref{Nepoinineq} hold in $B_x(r)$. Fix $s\in\mathbb{R}$,
$\delta,\tau\in(0,1)$. For any smooth positive solution $u$ of the
$f$-heat equation in $Q=B_x(r)\times(s-r^2,s)$, there exists a
constant $c=c(u)$ depending on $u$ such that for all $\lambda>0$,
\begin{equation*}
\begin{split}
\bar{\mu}(\{(z,t)\in R_+|\log u<-\lambda-c\})\leq&
C_0\lambda^{-1},\\
\bar{\mu}(\{(z,t)\in R_-|\log u>\lambda-c\})\leq& C_0\lambda^{-1},
\end{split}
\end{equation*}
where
$C_0=c_1(n,\delta,\tau)e^{c_2(A'+\sqrt{K})r+c_3A'\sqrt{K}r^2}V_f(B)r^2$.
Here $R_+=\delta B\times(s-\tau r^2,s)$ and $R_-=\delta
B\times(s-r^2,s-\tau r^2)$.
\end{lemma}
\begin{proof}
By shrinking the ball $B$ a little, we can assume that $u$ is a
positive solution in $B_x(r')\times(s-r^2,s)$ for some $r'>r$. Let
$\omega=-\log u$. Then for any nonnegative function $\psi\in
C_0(B_x(r'))$, we have
\[
\partial_t\int \psi^2\omega d\mu=-\int\psi^2 u^{-1}\Delta_fu d\mu
=\int\left[-\psi^2|\nabla\omega|^2-2\psi\nabla\omega\cdot\nabla\psi\right]d\mu.
\]
By Cauchy-Schwarz inequality $2|ab|\leq 1/2 a^2+2b^2$, we obtain
\[
\partial_t\int \psi^2\omega d\mu+1/2\int^2|\nabla\omega|^2d\mu\leq
2||\nabla \psi||_{\infty}^2V_f(\rm{supp}(\psi)).
\]
Fix $0<\delta<1$ and define function $\xi$ such that $\xi=1$ on
$[0,\delta]$, $\xi(t)=\frac{1-t}{1-\delta}$ on $[\delta,1]$ and
$\xi=0$ on $[1,\infty)$. We set $\Psi=\xi(\rho(x,\cdot)/r)$.
Clearly, we can apply the above to $\psi=\Psi$. Then Lemma
\ref{weigNP} can be applied with $\Psi^2$ as a weight function.
Thus, we have
\[
\int |\nabla\omega|^2\Psi^2d\mu\geq \left(c_\delta r^2
e^{c_2(A'+\sqrt{K})r+c_3A'\sqrt{K}r^2}\right)^{-1}\int|\omega-W|^2\Psi^2d\mu,
\]
where $W:=\int\Psi^2\omega d\mu/\int\Psi^2d\mu$. Noticing that $\int
\Psi^2$ is comparable to $V_f$, so
\[
\partial_tW+C^{-1}_1\int_{\delta B}|\omega-W|^2\leq C_2,
\]
where
$C_1=C(\delta,\tau)e^{c_2(A'+\sqrt{K})r+c_3A'\sqrt{K}r^2}r^2V_f$ and
$C_2=C(\delta,\tau)r^{-2}$. Letting $s'=s-\tau r^2$, the above
inequality can be written as
\[
\partial_t\overline{W}+C^{-1}_1\int_{\delta B}|\overline{\omega}-\overline{W}|^2\leq 0,
\]
where $\overline{\omega}(z,t)=\omega(z,t)-C_2(t-s')$ and
$\overline{W}(z,t)=W(z,t)-C_2(t-s')$.

Now we set
\[
c=W(s')=\overline{W}(s'),
\]
and for $\lambda>0$, $s-r^2<t<s$, we define two sets
\[
\Omega^{+}_t(\lambda)=\{z\in \delta B,\bar{\omega}(z,t)>c+\lambda\}
\quad\mathrm{and}\quad \Omega^{-}_t(\lambda)=\{z\in \delta
B,\bar{\omega}(z,t)<c-\lambda\}.
\]
Then if $t>s'$, we have
\[
\overline{\omega}(z,t)-\overline{W}(t)\geq
\lambda+c-\overline{W}(t)>\lambda
\]
in $\Omega^{+}_t(\lambda)$, since $c=\overline{W}(s')$ and
$\partial_t\overline{W}\leq 0$. Similarly, if $t<s'$, then we have
\[
\overline{\omega}(z,t)-\overline{W}(t)\leq-\lambda+c-\overline{W}(s')<-\lambda
\]
in $\Omega^{-}_t(\lambda)$. Hence, if $t>s'$, we obtain
\[
\partial_t\overline{W}(t)+C^{-1}_1|\lambda+c-\overline{W}(t)|^2\mu(\Omega^{+}_t(\lambda))\leq 0
\]
and namely,
\[
-C_1\partial_t(|\lambda+c-\overline{W}(t)|^{-1})\geq\mu(\Omega^{+}_t(\lambda)).
\]
Integrating from $s'$ to $s$,
\[
\bar{\mu}(\{(z,t)\in R_+,\overline{\omega}>c+\lambda\})\leq
C_1\lambda^{-1}.
\]
Recalling that $-\log u=\omega=\overline{\omega}+C_2(t-s')$, hence
\[
\bar{\mu}(\{(z,t)\in R_+,\log u<-\lambda-c\})\leq
(\max\{C_1,C_2r^4V_f\})\lambda^{-1}.
\]
This gives the first estimate of the lemma. The second estimate
follows from a similar argument by working with $\Omega^{-}_t$ and
$t<s'$.
\end{proof}

Thirdly, in order to finish the proof of Theorem \ref{Harthm}, we
need the following elementary lemma. This is in fact an iterated
procedure. We let $R_\sigma$, $0<\sigma\leq 1$ be a collection of
subset for some space-time endowed with the measure $d\bar{\mu}$
such that $R_{\sigma'}\subset R_\sigma$ if $\sigma'\leq\sigma$.
Indeed, $R_\sigma$ will be one of the collections $Q_\delta$ or
${Q'}_\delta$.
\begin{lemma}\label{crilem2}
Let $\gamma$, $C$, $1/2\leq\delta<1$, $p_1<p_0\leq \infty$ be
positive constants, and let $\varphi$ be a positive smooth function
on $R_1$ such that
\[
||\varphi||_{p_0,R_{\sigma'}}
\leq\{C(\sigma-\sigma')^{-\gamma}V_f^{-1}(R_1)\}^{1/p-1/{p_0}}||\varphi||_{p,R_{\sigma}}
\]
for all $\sigma$, $\sigma'$, $p$ satisfying
$1/2\leq\delta\leq\sigma'<\sigma\leq 1$ and $0<p\leq p_1<p_0$.
Besides, if $\varphi$ also satisfies
\[
Vol_f(\{z\in R_1,\ln \varphi>\lambda\})\leq C V_f(R_1)\lambda^{-1}
\]
for all $\lambda>0$, then we have
\[
||\varphi||_{p_0,R_\delta}\leq (V_f(R_1))^{1/{p_0}}e^{C_1(1+C^3)},
\]
where $C_1$ depends only on $\gamma$, $\delta$ and a positive lower
bound on $1/{p_1}-1/{p_0}$.
\end{lemma}
\begin{proof}
Without loss of generality we may assume that $Vol_f(R_1)=1$. Let
\[
\zeta=\zeta(\sigma):=\ln(||\varphi||_{p_0,R_\sigma}),\quad
\delta\leq\sigma<1.
\]
We divide $R_\sigma$ into two sets: $\{\ln\varphi>\zeta/2\}$ and
$\{\ln\varphi\leq\zeta/2\}$. Then
\begin{equation*}
\begin{aligned}
||\varphi||_{p,R_\sigma}&\leq ||\varphi||_{p_0,R_\sigma}\cdot
V_f(\{z\in R_\sigma,\ln\varphi>\zeta/2\})^{1/p-1/{p_0}}+e^{\zeta/2}\\
&\leq
e^\zeta\left(\frac{2C}{\zeta}\right)^{1/p-1/{p_0}}+e^{\zeta/2},
\end{aligned}
\end{equation*}
where $p<p_0$. Here in the first inequality we used the H\"{o}lder
inequality, and in the second inequality we used the second
assumption of lemma. In the following we want to choose $p$ such
that the last two terms in above are equal, and $0<p\leq p_1$. This
is possible if
\[
(1/p-1/{p_0})^{-1}=(2/\zeta)\ln\left(\frac{\zeta}{2C}\right)\leq(1/{p_1}-1/{p_0})^{-1}
\]
and the last inequality is satisfied as long as
\[
\zeta\geq C_2C,
\]
where $C_2$ depends only on a positive lower bound on
$1/{p_1}-1/{p_0}$. Now we assume $p$ and $\zeta$ have been chosen as
above. Then we obtain
\[
||\varphi||_{p,R_\sigma}\leq 2e^{\zeta/2}.
\]
Using the first assumption of the lemma and the definition of
$\kappa$, we have
\begin{equation*}
\begin{aligned}
\kappa(\sigma')
&\leq\ln\left\{2\left(C(\sigma-\sigma')^{-\gamma}\right)^{1/p-1/{p_0}}e^{\zeta/2}\right\}\\
&=(1/p-1/{p_0})\ln[C(\sigma-\sigma')^{-\gamma}]+\ln 2+\zeta/2
\end{aligned}
\end{equation*}
for any $\delta\leq\sigma'<\sigma\leq1$. According to our choice of
$p$ above, we get
\[
\kappa(\sigma') \leq \frac\zeta
2\left\{\frac{\ln[C(\sigma-\sigma')^{-\gamma}]}{\ln(\zeta/C)}+\frac{2\ln
2}{\zeta}+1\right\}.
\]
Here, on one hand, if we choose
\[
\zeta\geq16C^3(\sigma-\sigma')^{-2\gamma}+8\ln 2,
\]
then the above inequality becomes
\[
\zeta(\sigma')\leq\frac 34\zeta.
\]
On the other hand, if the assumption of $\kappa$ above in not
satisfied, we can have
\[
\zeta(\sigma')\leq\zeta(\sigma)\leq
C_2C+16C^3(\sigma-\sigma')^{-2\gamma}+8\ln 2.
\]
Therefore, in any case
\[
\zeta(\sigma')\leq\frac
34\zeta(\sigma)+C_3(1+C^3)(\sigma-\sigma')^{-2\gamma}.
\]
for any $\delta\leq\sigma'<\sigma\leq1$, where $C_3=C_2+16+8\ln 2$.
From this, an routine iteration (see \cite{[Moser2]}, page 733)
yields
\[
\zeta(\delta)\leq C_4(1-\delta)^{-2\gamma}(1+C^3),
\]
where $C_4$ depends on $C_3$ and $\gamma$. This completes the proof
of the lemma.
\end{proof}

Now, applying Lemma \ref{crilem1}, Lemma \ref{crilem2} and
Proposition \ref{PoinDou2}, we get the following Harnack inequality.
\begin{theorem}\label{wHarn}
Let $(M,g,e^{-f}dv)$ be an $n$-dimensional complete noncompact
smooth metric measure space. Assume that \eqref{voldop} and
\eqref{Nepoinineq} hold in $B_x(r)$. Fix $\tau\in(0,1)$ and
$0<p_0<1+\nu/2$. For any $s\in\mathbb{R}$ and
$0<\varepsilon<\eta<\delta<1$, any smooth positive solution $u$ of
the $f$-heat equation in the cylinder $Q=B_x(r)\times(s-r^2,s)$
satisfies
\[
\parallel u\parallel_{p_0, Q_-}\leq
(r^2V_f)^{\frac{1}{p_0}}e^{c_1F(r)}\inf_{Q_+} u,
\]
where $c_1=c_1(n,\varepsilon,\eta,\delta,p_0)$ and
$F(r)=e^{c_2(A'+\sqrt{K})r+c_3A'\sqrt{K}r^2}$, $A'=A'(x,r)$. Hence
we have
\[
\sup_{Q_-} u\leq e^{c_4F(r)}\inf_{Q_+} u,
\]
where $c_4=c_4(n,\varepsilon,\eta,\delta)$.
\end{theorem}

\begin{proof}[Proof of Theorem \ref{wHarn}]
We let $u$ be a positive solution to the $f$-heat equation in $Q$.
Let also $\delta,\tau\in(0,1)$ be fixed. Using Proposition
\ref{PoinDou2} and Lemma \ref{crilem1}, we see that Lemma
\ref{crilem2} can be applied to $e^cu$ (resp. $e^{-c}u^{-1}$), where
$c=c(u)$ is defined as in Lemma \ref{crilem1}, with
\[
R_\sigma=\sigma\delta B\times(s-r^2,s-\tau r^2-\sigma\tau r^2) \quad
(\mathrm{resp.}\,\, R_\sigma=\sigma\delta B\times(s-\sigma\tau r^2,
s))
\]
and $0<p_1=p_0/2<p_0<1+\nu/2$ (resp. $0<p_1=1<p_0=\infty$). Hence
for any $0<\varepsilon<\eta<\delta<1$ and $Q_{-}$, $Q_{+}$ as
defined as above, we have
\[
e^c\parallel u\parallel_{p_0, Q_-}\leq (r^2V_f)^{1/{p_0}}e^{c_1F(r)}
\]
and
\[
e^{-c}\sup_{Q_+}\{u^{-1}\}\leq e^{c_4F(r)},
\]
where $c_1=c_1(n,\varepsilon,\eta,\delta,p_0)$,
$c_4=c_4(n,\varepsilon,\eta,\delta)$ and
$F(r)=e^{c_2(A'+\sqrt{K})r+c_3A'\sqrt{K}r^2}$. The theorem follows
from this and Proposition \ref{PoinDouHarn}.
\end{proof}

Finally, we finish the proof of Theorem \ref{Harthm} by applying the
standard chain argument to Theorem \ref{wHarn}.

\begin{proof}[Proof of theorem \ref{Harthm}]
Let $(t_-, x_-)\in Q_-$, $(t_+, x_+)\in Q_+$, and let
$\tau=t_+-t_-$. Notice that $\tau\sim r^2$ and $d=d(x_-,x_+)<r$. Let
$t_i=t_-+\frac{i\tau}{N}$ and $x_i\in\frac{1+\delta}{2}B$ for $0\leq
i\leq N$, such that $x_0=x_-$, $x_N=x_+$, and $d(x_i, x_{i+1})\leq
C_{\delta}\frac{d}{N}$. Choose $N$ to be the smallest number such
that
\[
N\geq C_{\varepsilon,\eta,\delta}(A'+\sqrt{K})^2 r^2,
\]
where $A'=A'(x,r+1)$, applying Theorem \ref{wHarn} with
$r'=(\frac{\tau}{N})^{\frac{1}{2}}$, then we have
\begin{equation*}
\begin{aligned}
u(t_-, x_-)&\leq e^{c_4F(r')(N+1)}u(t_+, x_+)\\
&\leq e^{c_4F\left(\frac{1}{C(A'+\sqrt{K})}\right)(N+1)}u(t_+, x_+)\\
&\leq \exp\left[c(A'+\sqrt{K})^2 r^2+c\right]u(t_+, x_+),
\end{aligned}
\end{equation*}
where $c$ depends on $n$, $\varepsilon$, $\eta$ and $\delta$. This
finishes the proof of Theorem \ref{Harthm}.
\end{proof}


\section{Gaussian upper and lower bounds of the $f$-heat kernel}\label{sec3}

In this section, following the arguments in \cite{[Saloff2]}, we
derive Gaussian upper and lower bounds for the $f$-heat kernel on
smooth metric measure spaces. The upper bound estimate follows from
the $f$-mean value inequality in Proposition \ref{PoinDouHarn} and a
weighted version of Davies integral estimate (see \cite{[WuWu]}).
The lower bound estimate follows from the local Harnack inequality
in Section \ref{sec3v}.

Let us first state the weighted Davies integral estimate, see
\cite{[WuWu]} for the proof,
\begin{lemma}\label{lemm3.3}
Let $(M,g,e^{-f}dv)$ be an $n$-dimensional complete smooth metric
measure space. Let $\lambda_1(M)\geq0$ be the bottom of the
$L_f^2$-spectrum of the $f$-Laplacian on $M$. Assume that $B_1$ and
$B_2$ are bounded subsets of $M$. Then
\begin{equation}\label{Dvpp}
\int_{B_1}\int_{B_2}H(x,y,t)d\mu(x)d\mu(y)\leq
V^{1/2}_f(B_1)V^{1/2}_f(B_2)\exp\left(-\lambda_1(M)t-\frac{d^2(B_1,B_2)}{4t}\right),
\end{equation}
where $d(B_1,B_2)$ denotes the distance between the sets $B_1$ and
$B_2$.
\end{lemma}

\begin{proof}[Proof of upper bound estimate in Theorem \ref{Main1}]
For $x\in B_o(R/2)$, denote $u(y,s)=H(x,y,s)$.  Assume $t\geq
r^2_2$, applying Proposition \ref{PoinDouHarn} to $u$, we have
\begin{equation}\label{itegA}
\begin{aligned}
\sup_{(y,s)\in Q_\delta}H(x,y,s)&\leq
\frac{c_1e^{c_2A+c_3(1+A)\sqrt{K}r_2}}{r^2_2V_f(B_2)}
\cdot\int^t_{t-1/4r^2_2}\int_{B_2}H(x,\zeta,s) d\mu(\zeta)ds\\
&=\frac{c_1e^{c_2A+c_3(1+A)\sqrt{K}r_2}}{4V_f(B_2)}\cdot\int_{B_2}H(x,\zeta,s')
d\mu(\zeta)
\end{aligned}
\end{equation}
for some $s'\in(t-1/4r^2_2, t)$, where $Q_\delta=B_y(\delta
r_2)\times(t-\delta r^2_2, t)$ with $0<\delta<1/4$, and
$B_2=B_y(r_2)\subset B_o(R)$ for $y\in B_o(R/2)$, $A=A(x,R)\leq
A(o,2R)$. Applying Proposition \ref{PoinDouHarn} and the same
argument to the positive solution
\[
v(x,s)=\int_{B_2}H(x,\zeta,s) d\mu(\zeta)
\]
of the $f$-heat equation, for the variable $x$ with $t\geq r^2_1$,
we also get
\begin{equation}\label{itegB}
\begin{aligned}
\sup_{(x,s)\in \bar{Q}_\delta}\int_{B_2}H(x,\zeta,s) d\mu(\zeta)
&\leq\frac{c_1e^{c_2A+c_3(1+A)\sqrt{K}r_1}}{r^2_1V_f(B_1)}\cdot\int^t_{t-1/4r^2_1}
\int_{B_1}\int_{B_2}H(\xi,\zeta,s)d\mu(\zeta)d\mu(\xi)ds\\
&=\frac{c_1e^{c_2A+c_3(1+A)\sqrt{K}r_1}}{4V_f(B_1)}\cdot
\int_{B_1}\int_{B_2}H(\xi,\zeta,s'')d\mu(\zeta)d\mu(\xi)
\end{aligned}
\end{equation}
for some $s''\in(t-1/4r^2_1, t)$, where $\bar{Q}_\delta=B_x(\delta
r_1)\times(t-\delta r^2_1, t)$ with $0<\delta<1/4$, and
$B_1=B_x(r_1)\subset B_o(R)$ for $x\in B_o(R/2)$. Now letting
$r_1=r_2=\sqrt{t}$ and combining \eqref{itegA} with \eqref{itegB},
the smooth $f$-heat kernel satisfies
\begin{equation}\label{heaup1}
H(x,y,t)\leq\frac{c_1e^{c_2A+c_3(1+A)\sqrt{Kt}}}{V_f(B_1)V_f(B_2)}\cdot
\int_{B_1}\int_{B_2}H(\xi,\zeta,s'')d\mu(\zeta)d\mu(\xi)
\end{equation}
for all $x,y\in B_o(R/2)$ and $0<t<R^2/4$. Using Lemma \ref{lemm3.3}
and noticing that $s''\in(\frac 34t, t)$, then \eqref{heaup1}
becomes
\begin{equation}\label{heaup2}
H(x,y,t)\leq\frac{c_1e^{c_2A+c_3(1+A)\sqrt{Kt}}}{V_f(B_x(\sqrt
{t}))^{1/2}V_f(B_y(\sqrt {t}))^{1/2}} \times\exp\left(-\frac
34\lambda_1t-\frac{d^2(B_1,B_2)}{4t}\right)
\end{equation}
for all $x,y\in B_o(R/2)$ and $0<t<R^2/4$. Notice that if
$d(x,y)\leq 2\sqrt{t}$, then $d(B_x(\sqrt {t}),B_y(\sqrt {t}))=0$
and hence
\[
-\frac{d^2(B_x(\sqrt{t}),B_y(\sqrt{t}))}{4t}=0\leq
1-\frac{d^2(x,y)}{4t},
\]
and if $d(x,y)>2\sqrt{t}$, then
$d(B_x(\sqrt{t}),B_y(\sqrt{t}))=d(x,y)-2\sqrt{t}$, and hence
\[
-\frac{d^2(B_x(\sqrt{t}),B_y(\sqrt{t}))}{4t}
=-\frac{(d(x,y)-2\sqrt{t})^2}{4t} \leq
-\frac{d^2(x,y)}{4(1+\epsilon)t}+C(\epsilon)
\]
for some constant $C(\epsilon)$, where $\epsilon>0$. Here if
$\epsilon \to 0$, then the constant $C(\epsilon)\to \infty$.
Therefore in any case, \eqref{heaup2} becomes
\[
H(x,y,t)\leq\frac{C(\epsilon)e^{c_2A+c_3(1+A)\sqrt{Kt}}}{V_f(B_x(\sqrt{t}))^{1/2}V_f(B_y(\sqrt{t}))^{1/2}}
\times\exp\left(-\frac
34\lambda_1t-\frac{d^2(x,y)}{4(1+\epsilon)t}\right)
\]
for all $x,y\in B_o(\frac 12R)$ and $0<t<R^2/4$.
\end{proof}

Moreover, in Theorem \ref{Main1}, if $K>0$. According to Lemma
\ref{lecomp1}, we know that
\[
V_f(B_x(\sqrt{t}))\leq
\frac{e^{(n-1+4A)\sqrt{K}(d(x,y)+\sqrt{t})}}{t^{n/2+2A}}
V_f(B_y(\sqrt{t}))
\]
for all $x,y\in B_o(\frac 14R)$ and $0<t<R^2/4$. Substituting this
into Theorem \ref{Main1} yields the following result.
\begin{corollary}\label{Mainco}
Let $(M,g,e^{-f}dv)$ be an $n$-dimensional complete noncompact
smooth metric measure space with $\mathrm{Ric}_f\geq -(n-1)K$ for
some constant $K>0$. For any point $o\in M$, $R>0$, $\epsilon>0$,
there exist constants $c_1(n,\epsilon)$, $c_2(n)$ and $c_3(n)$, such
that
\begin{equation}\label{uppfu2}
H(x,y,t)\leq\frac{c_1\,e^{c_2A+c_3(1+A)\sqrt{K}(d(x,y)+\sqrt{t})}}{V_f(B_x(\sqrt{t})\,\,t^{n/4+A}}
\times\exp\left(-\frac{d^2(x,y)}{(4+\epsilon)t}\right)
\end{equation}
for all $x,y\in B_o(\frac 14R)$ and $0<t<R^2/4$. Here
$\lim_{\epsilon\to 0}c_1(n,\epsilon)=\infty$.
\end{corollary}
When $K=0$, see the estimate in \cite{[WuWu]}.

\

Next we derive the lower bound estimate. First, from the Harnack
inequality in Theorem \ref{Harthm} we get the following estimate,
\begin{proposition}\label{coro4.1}
Under the same assumptions of Theorem \ref{Harthm}, there exists a
constant $c(n)$ such that, for any two positive solutions $u(x,s)$
and $u(y,t)$ of the $f$-heat equation in $B_o(R/2)\times(0,T)$,
$0<s<t<T$,
\[
\ln\left(\frac{u(x,s)}{u(y,t)}\right)\leq
c(n)\left[\left(A'^2+K+\frac{1}{R^2}+\frac{1}{s}\right)(t-s)+\frac{d^2(x,y)}{t-s}\right].
\]
\end{proposition}

\begin{proof}
Let $u(x,s)$ and $u(y,t)$ be two positive solutions to the $f$-heat
equation in $B_o(\delta R)\times(0,T)$, where $x,y\in B_o(\delta R)$
and $0<s<t<T$. Let $N$ be an integer, which will be chosen later. We
set $t_i=s+i(t-s)/N$. We remark that it is possible to find a
sequence of points $x_i\in\frac{1+\delta}{2}B$ such that $x_0=x$,
$x_N=y$ and $N d(x_i,x_{i+1})\geq C_\delta d(x,y)$. Now we choose
$N$ to be the smallest integer such that
\[
\tau/N\leq s/2,\quad \tau/N\leq C^{-1}_\delta R^2,\quad \tau=t-s
\]
and if $d(x,y)^2\geq \tau$,
\[
\tau/N\geq d(x,y)^2/N^2.
\]
Under the above conditions, we choose
\[
N=c_\delta\left(\frac{\tau}{R^2}+\frac{\tau}{s}+\frac{d(x,y)^2}{\tau}\right).
\]
Now we apply Theorem \ref{Harthm} to compare $u(x_i,t_i)$ with
$u(x_{i+1},t_{i+1})$ with $r'=(\tau/N)^{1/2}$. Therefore
\begin{equation*}
\begin{aligned}
\ln\left(\frac{u(x,s)}{u(y,t)}\right)&\leq
c_1\left[(A'^2+K)\frac\tau N+1\right]\cdot N\\
&\leq
c'_1\left[(A'^2+K)\tau+\frac{\tau}{R^2}+\frac{\tau}{s}+\frac{d(x,y)^2}{\tau}\right],
\end{aligned}
\end{equation*}
where $c'_1$ depends on $n$ and $\delta$, and $\tau=t-s$. Then the
conclusion follows by letting $\delta=1/2$.
\end{proof}

From Corollary \ref{coro4.1}, we get the following lower bound for
$f$-heat kernel,
\begin{theorem}\label{thm45}
Let $(M,g,e^{-f}dv)$ be an $n$-dimensional complete noncompact
smooth metric measure space with $\mathrm{Ric}_f\geq -(n-1)K$ for
some constant $K>0$. For any point $o\in M$ and $R>0$, there exist
constants $c_1(n)$, $c_2(n)$ and $c_3(n)$ such that
\begin{equation}\label{lowhe2x}
H(x,y,t)\geq \frac{c_1e^{-c_2(A'^2+K)t}}{V_f(B_x(\sqrt{t}))}
\times\exp\left(-\frac{d^2(x,y)}{c_3t}\right),
\end{equation}
for all $x,y\in B_o(\frac 12R)$ and $0<t<R^2/4$.
\end{theorem}

\begin{proof}
[Proof of Theorems \ref{thm45} and the second part of Theorem
\ref{Main1}] Let $u(y,t)=H(x,y,t)$ with $x$ fixed and $s=t/2$ in
Proposition \ref{coro4.1} and then we get
\begin{equation}\label{lowheax}
H(x,y,t)\geq H(x,x,t/2)\times\exp
\left[-c_1\left((A'^2+K)t+1+\frac{t}{R^2}+\frac{d^2(x,y)}{t}\right)\right]
\end{equation}
for all $x,y\in B_o(\frac 12R)$ and $0<t<\infty$.

In the following we will show that Moser's Harnack inequality leads
to a lower bound of the on-diagonal $f$-heat kernel $H(x,x,t)$.
Indeed we define
\begin{equation*}
u(y,t)=\left\{
\begin{aligned}
P_t\phi(y)\quad\mathrm{if}\quad t>0 \\
\phi(y)\quad \mathrm{if} \quad t\leq 0,
\end{aligned} \right.
\end{equation*}
where $P_t=e^{t\Delta_f}$ is the heat semigroup of $\Delta_f$, and
$\phi$ is a smooth function such that $0\leq\phi\leq1$, $\phi=1$ on
$B=B_x(\sqrt{t})$ and $\phi=0$ on $M\setminus2B$.

$u(y,t)$ satisfies $(\partial_t-\Delta_f)u=0$ on
$B\times(-\infty,\infty)$. Applying the local Harnack inequality,
first to $u$, and then to the $f$-heat kernel $(y,s)\to H(x,y,s)$,
we have
\begin{equation*}
\begin{aligned}
1=u(x,0)&\leq \exp\left\{c_1\left[(A'^2+K)t+1\right]\right\}\,u(x,t/2)\\
&=\exp\left\{c_1\left[(A'^2+K)t+1\right]\right\}\int_{B(x,\sqrt{t})} H(x,y,t/2)\phi(y)d\mu(y)\\
&\leq \exp\left\{c_1\left[(A'^2+K)t+1\right]\right\}\int_{B(x,2\sqrt{t})} H(x,y,t/2)d\mu(y)\\
&\leq
\exp\left\{2c_1\left[(A'^2+K)t+1\right]\right\}V_f(B_x(2\sqrt{t}))H(x,x,t).
\end{aligned}
\end{equation*}
From this, we have
\[
H(x,x,t/2)\geq
V^{-1}_f(B_x(\sqrt{2t}))\exp\left[-c_1\Big((A'^2+K)t+2\Big)\right]
\]
for $0<\sqrt{t}<R/2$. Since \eqref{voldop} implies
\[
V_f(B_x(\sqrt{2t}))\leq V_f(B_x(2\sqrt{t}))\leq
c_1e^{c_2(A'+\sqrt{K})\sqrt{t}+c_3A'\sqrt{K}t}V_f(B_x(\sqrt{t})),
\]
we then obtain
\[
H(x,x,t/2)\geq
V^{-1}_f(B_x(\sqrt{t}))c_4\exp\left[-c_5\Big((A'^2+K)t+1\Big)\right]
\]
for $0<\sqrt{t}<R/2$. Plugging this into \eqref{lowheax} yields
\eqref{lowhe2x}.
\end{proof}


\section{$L_f^1$-Liouville theorem}\label{sec4}
In this section, inspired by the work of P. Li \cite{[Li0]}, we
prove a Liouville theorem for $f$-subharmonic functions, and a
uniqueness result for solutions of $f$-heat equation, by applying
the $f$-heat kernel upper bound estimates. Our results not only
extend the classical $L^1$-Liouville theorems proved by P. Li
\cite{[Li0]}, but also generalize the weighted versions in
\cite{[LD]}, \cite{[Wu2]}, and \cite{[WuWu]}.

Firstly we prove an $L_f^1$-Liouville theorem for $f$-harmonic
functions when the Bakry-\'{E}mery Ricci curvature is bounded below
and $f$ is of linear growth.

\begin{theorem}\label{Main2}
Let $(M,g,e^{-f}dv)$ be an $n$-dimensional complete noncompact
smooth metric measure space with $\mathrm{Ric}_f\geq -(n-1)K$ for
some constant $K>0$. Assume there exist nonnegative constants $a$
and $b$ such that
\[
|f|(x)\leq ar(x)+b\,\,\, {for}\,\, {all}\,\, x\in M,
\]
where $r(x)$ is the geodesic distance function to a fixed point
$o\in M$. Then any nonnegative $L_f^1$-integrable $f$-subharmonic
function must be identically constant. In particular, any
$L_f^1$-integrable $f$-harmonic function must be identically
constant.
\end{theorem}

\begin{proof}[Sketch proof of Theorem \ref{Main2}]
We first show that the assumptions of Theorem \ref{Main2} imply the
integration by parts formula
\[
\int_M {\Delta_f}_y H(x,y,t)h(y)d\mu(y) =\int_M H(x,y,t)\Delta_f\
h(y)d\mu(y)
\]
for any nonnegative $L_f^1$-integrable $f$-subharmonic function $h$.
This can be proved by our upper bound of $f$-heat kernel in Theorem
\ref{Main1}. Then following the arguments of \cite{[WuWu]}, applying
the regularity theory of $f$-harmonic functions, we obtain the
$L_f^1$-Liouville result. See the proof of Theorem 1.5 in
\cite{[WuWu]} for the details.
\end{proof}

Now we are ready to check the integration by parts formula, similar
to the proof of Theorem 4.3 in \cite{[WuWu]},
\begin{proposition}\label{integra}
Under the same assumptions of Theorem \ref{Main2}, for any
nonnegative $L_f^1$-integrable $f$-subharmonic function $h$, we have
\[
\int_M {\Delta_f}_y H(x,y,t)h(y)d\mu(y) =\int_M
H(x,y,t)\Delta_fh(y)d\mu(y).
\]
\end{proposition}

\begin{proof}[Proof of Proposition \ref{integra}]
By the Green formula on $B_o(R)$, we have
\begin{equation*}
\begin{aligned}
&\left|\int_{B_o(R)}{\Delta_f}_y H(x,y,t)h(y)d\mu(y)
-\int_{B_o(R)}H(x,y,t)\Delta_f h(y)d\mu(y)\right|\\
\leq&\int_{\partial B_o(R)}H(x,y,t)|\nabla h|(y)d\mu_{\sigma,R}(y)
+\int_{\partial B_o(R)}|\nabla H|(x,y,t)h(y)d\mu_{\sigma,R}(y),
\end{aligned}
\end{equation*}
where $d\mu_{\sigma,R}$ denotes the weighted area measure induced by
$d\mu$ on $\partial B_o(R)$. In the following we will show that the
above two boundary integrals vanish as $R\to \infty$.

Consider a large $R$ and assume $x\in B_o(R/8)$. By Proposition
\ref{PoinDouHarn}, we have the $f$-mean value inequality
\begin{equation}
\begin{aligned}\label{mjbd}
\sup_{B_o(R)}h(x)&\leq c_1e^{c_2(aR+b)+c_3(1+aR+b)\sqrt{K}R}V_f^{-1}(2R)\int_{B_o(2R)}h(y)d\mu(y)\\
&\leq Ce^{\alpha(1+K)R^2}V_f^{-1}(2R)\int_{B_o(2R)}h(y)d\mu(y),
\end{aligned}
\end{equation}
where constants $C$ and $\alpha$ depend on $n$, $a$ and $b$. Let
$\phi(y)=\phi(r(y))$ be a nonnegative cut-off function satisfying
$0\leq\phi\leq1$, $|\nabla\phi|\leq\sqrt{3}$ and $\phi(r(y))=1$ on
$B_o(R+1)\backslash B_o(R)$, $\phi(r(y))=1$ on
$B_o(R-1)\cup(M\backslash B_o(R+2))$. Since $h$ is $f$-subharmonic,
by the integration by parts formula and Cauchy-Schwarz inequality,
we have
\begin{equation*}
\begin{split}
0\leq\int_M\phi^2h\Delta_fh d\mu =&-2\int_M \phi
h\langle\nabla\phi\nabla h\rangle d\mu
-\int_M \phi^2|\nabla h|^2d\mu\\
\leq&\ 2\int_M |\nabla\phi|^2h^2d\mu -\frac12\int_M \phi^2|\nabla
h|^2d\mu.
\end{split}
\end{equation*}
Then using the definition of $\phi$ and \eqref{mjbd}, we have that
\begin{equation*}
\begin{aligned}
\int_{B_o(R+1)\backslash B_o(R)}|\nabla h|^2d\mu\leq&
4\int_M|\nabla \phi|^2h^2d\mu\\
\leq&12\int_{B_o(R+2)}h^2d\mu\\
\leq&12\sup_{B_o(R+2)}h\cdot\|h\|_{L^1(\mu)}\\
\leq&\frac{Ce^{\alpha(1+K)(R+2)^2}}{V_f(2R+4)}
\cdot\|h\|_{L^1(\mu)}^2.
\end{aligned}
\end{equation*}
On the other hand, the Cauchy-Schwarz inequality also implies
\[
\int_{B_o(R+1)\backslash B_o(R)}|\nabla h|d\mu\leq
\left(\int_{B_o(R+1)\backslash B_o(R)}|\nabla h|^2d\mu\right)^{1/2}
\cdot [V_f(R+1)\backslash V_f(R)]^{1/2}.
\]
Combining the above two inequalities we get
\begin{equation}\label{guji5}
\int_{B_o(R+1)\backslash B_o(R)}|\nabla h|d\mu\leq C_1e^{\alpha(1+K)
R^2} \cdot\|h\|_{L^1(\mu)},
\end{equation}
where $C_1=C_1(n,a,b,K)$.

\vspace{0.5em}

We now estimate the $f$-heat kernel $H(x,y,t)$. Recall that, by
letting $\epsilon=1$ in Corollary \ref{Mainco}, the $f$-heat kernel
$H(x,y,t)$ satisfies
\begin{equation} \label{guji2}
\begin{split}
H(x,y,t)\leq&\frac{c_1\,e^{c_2A+c_3(1+A)\sqrt{K}(d(x,y)+\sqrt{t})}}{V_f(B_x(\sqrt{t})\,\,t^{n/4+A}}
\times\exp\left(-\frac{d^2(x,y)}{5t}\right)\\
\leq&\frac{c_4\,e^{c_5R}}{V_f(B_x(\sqrt{t}))t^{c_7(R+1)}}
\exp\left[c_6\sqrt{K}(1+R)(d(x,y)+\sqrt{t})-\frac{d^2(x,y)}{5t}\right]
\end{split}
\end{equation}
for any $x,y\in B_o(R/2)$ and $0<t<R^2/4$, where $c_4$, $c_5$, $c_6$
and $c_7$ are all constants depending only on $n$, $a$ and $b$.
Together with \eqref{guji5} we get
\begin{equation*}
\begin{aligned}
J_1:=&\int_{B_o(R+1)\backslash B_o(R)}H(x,y,t)|\nabla g|(y)d\mu(y)\\
\leq&\sup_{y\in {B_o(R+1)\backslash B_o(R)} }H(x,y,t)\cdot
\int_{B_o(R+1)\backslash B_o(R)}|\nabla g|d\mu\\
\leq&\frac{C_2\|g\|_{L^1(\mu)}}{V_f(B_x(\sqrt{t}))t^{c_7(R+2)}}
\times\exp\left[c_5R-\frac{(R-d(o,x))^2}{5t}+c_9\sqrt{K}(R+2)(R+1+d(o,x)+\sqrt{t})\right],
\end{aligned}
\end{equation*}
where $C_2=C_2(n,a,b,K)$. Notice that
\[
t^{-c_7(R+2)}=e^{-c_7(R+2)\ln t}\leq e^{c_7(R+2)\frac 1t}\quad
\mathrm{when}\quad t\to 0.
\]
Thus, for $T$ sufficiently small and for all $t\in (0,T)$ there
exists a constant $\beta>0$ such that
\[
J_1\leq \frac{C_3\|g\|_{L^1(\mu)}}{V_f(B_x(\sqrt{t}))}
\times\exp\left(-\beta R^2+c\frac{d^2(o,x)}{t}\right),
\]
where $C_3=C_3(n,a,b,K)$. Therefore for all $t\in(0,T)$ and all
$x\in M$, $J_1\rightarrow 0$ as $R\rightarrow\infty$.\\

By a similar argument, we can show that
\[
\int_{B_o(R+1)\backslash B_o(R)}|\nabla H|(x,y,t)h(y)d\mu\rightarrow
0
\]
as $R\rightarrow\infty$. We first estimate $\int_{B_o(R+1)\backslash
B_o(R)}|\nabla H|(x,y,t)d\mu$.
\begin{equation*}
\begin{aligned}
\int_M\phi^2(y)|\nabla H|^2(x,y,t)d\mu
&=-2\int_M\big\langle H(x,y,t)\nabla\phi(y),\phi(y)\nabla H(x,y,t)\big\rangle d\mu\\
&\quad-\int_M\phi^2(y)H(x,y,t)\Delta_f H(x,y,t)d\mu\\
&\leq2\int_M|\nabla\phi|^2(y)H^2(x,y,t)d\mu+\frac 12\int_M\phi^2(y)|\nabla H|^2(x,y,t)d\mu\\
&\quad-\int_M\phi^2(y)H(x,y,t)\Delta_f H(x,y,t)d\mu,
\end{aligned}
\end{equation*}
which implies
\begin{equation}
\begin{aligned}\label{cutfuest}
&\int_{B_o(R+1)\backslash B_o(R)}|\nabla H|^2\\
&\leq\int_M\phi^2(y)|\nabla H|^2(x,y,t)\\
&\leq4\int_M|\nabla\phi|^2H^2-2\int_M\phi^2H\Delta_f H\\
&\leq12\int_{B_o(R+2)\backslash B_o(R-1)}H^2+2\int_{B_o(R+2)\backslash B_o(R-1)}H|\Delta_f H|\\
&\leq12\int_{B_o(R+2)\backslash B_o(R-1)}H^2
+2\left(\int_{B_o(R+2)\backslash B_o(R-1)}H^2\right)^{\frac 12}
\left(\int_M(\Delta_f H)^2\right)^{\frac 12}.
\end{aligned}
\end{equation}
Notice that by Theorem 4.1 in \cite{[WW]}, if $\mathrm{Ric}_f\geq
-(n-1)K$, then
\[
V_f(B_o(R))\leq A+B\exp\Big[\frac{(n-1)K}{2}R^2\Big]
\]
for all $R>1$, so we have
\begin{equation}\label{integ}
\int^{\infty}_1\frac{R}{\log V_f(B_o(R))}dR=\infty.
\end{equation}
By Theorem 3.13 in \cite{[Grig3]}, $(M,g,e^{-f}dv)$ is
stochastically complete, i.e.,
\begin{equation}\label{stoch}
\int_MH(x,y,t)e^{-f}dv(y)=1.
\end{equation}
Using \eqref{guji2} and \eqref{stoch}, we get
\begin{equation}
\begin{aligned}\label{cutfuest2}
\int_{B_o(R+2)\backslash B_o(R-1)}H^2(x,y,t)d\mu
&\leq \sup_{y\in B_o(R+2)\backslash B_o(R-1)}H(x,y,t)\\
&\leq\frac{c_4}{V_f(B_x(\sqrt{t}))t^{c_7(R+3)}}\times\exp\left[-\frac{(R-1-d(o,x))^2}{5t}\right]\\
&\quad\times\exp\left[c_5(R+2)+c_6\sqrt{K}(3+R)(R+2-d(o,x)+\sqrt{t})\right]\\
&=\frac{c_4}{V_f(B_x(\sqrt{t}))}\times\exp\left[-\frac{(R-1-d(o,x))^2}{5t}+c_7(R+3)\ln \frac 1t\right]\\
&\quad\times\exp\left[c_5(R+2)+c_6\sqrt{K}(3+R)(R+2-d(o,x)+\sqrt{t})\right].
\end{aligned}
\end{equation}
From (4.7) in \cite{[WuWu]}, there exists a constant $C>0$ such that
\begin{equation}\label{Lapest}
\int_M(\Delta_f H)^2(x,y,t)d\mu\leq\frac{C}{t^2}H(x,x,t).
\end{equation}
Combining  \eqref{cutfuest}, \eqref{cutfuest2} and \eqref{Lapest},
we obtain
\begin{equation*}
\begin{aligned}
\int_{B_o(R+1)\backslash B_o(R)}|\nabla H|^2d\mu
&\leq C_4\left[V^{-1}_f+t^{-1}V^{-\frac 12}_f H^{\frac 12}(x,x,t)\right]\\
&\quad\times\exp\left[-\frac{(R-1-d(o,x))^2}{10t}+c_7(R+3)\ln \frac 1t\right]\\
&\quad\times\exp\left[c_5R+c_6\sqrt{K}(3+R)(R+2-d(o,x)+\sqrt{t})\right]
\end{aligned}
\end{equation*}
where $V_f=V_f(B_x(\sqrt{t}))$ and $C_4=C_4(n,a,b)$. Hence we get
\begin{equation}
\begin{aligned}\label{gujiest2}
\int_{B_o(R+1)\backslash B_o(R)}|\nabla H|d\mu
&\leq\left[V_f(B_o(R+1))\backslash V_f(B_o(R))\right]^{1/2}
\times\left[\int_{B_o(R+1)\backslash B_o(R)}|\nabla H|^2d\mu\right]^{1/2}\\
&\leq C_4V^{1/2}_f(B_o(R+1))
\left[V^{-1}_f+t^{-1}V^{-\frac 12}_f H^{\frac 12}(x,x,t)\right]^{1/2}\\
&\quad\times\exp\left[-\frac{(R-1-d(o,x))^2}{20t}+\frac{c_7}{2}(R+3)\ln \frac 1t\right]\\
&\quad\times\exp\left[c_5R+c_6\sqrt{K}(3+R)(R+2-d(o,x)+\sqrt{t})\right].
\end{aligned}
\end{equation}
Therefore, by \eqref{mjbd} and \eqref{gujiest2}, we obtain
\begin{equation*}
\begin{aligned}
J_2=&\int_{B_o(R+1)\backslash B_o(R)}|\nabla H(x,y,t)|h(y)d\mu(y)\\
\leq&\sup_{y\in B_o(R+1)\backslash B_o(R)}h(y)\cdot
\int_{B_o(R+1)\backslash B_o(R)}|\nabla H(x,y,t)|d\mu(y)\\
\leq& \frac{C_5\|g\|_{L^1(\mu)}}{V^{1/2}_f(B_o(2R+2))}
\cdot\left[V^{-1}_f+t^{-1}V^{-\frac 12}_f H^{\frac 12}(x,x,t)\right]^{1/2}\\
&\quad\times\exp\left[\alpha(1+K)(R+1)^2-\frac{(R-1-d(o,x))^2}{20t}+\frac{c_7}{2}(R+3)\ln \frac 1t\right]\\
&\quad\times\exp\left[c_5R+c_6\sqrt{K}(3+R)(R+2-d(o,x)+\sqrt{t})\right],
\end{aligned}
\end{equation*}
where $C_5=C_5(n,a,b)$. Similar to the case of $J_1$, we choose $T$
sufficiently small, then for all $t\in(0,T)$ and all $x\in M$,
$J_2\rightarrow 0$ when $R\rightarrow\infty$.\\

Now by the mean value theorem, for any $R>0$ there exists
$\bar{R}\in (R,R+1)$ such that
\begin{equation*}
\begin{aligned}
J&=\int_{\partial B_o(\bar{R})}\left[H(x,y,t)|\nabla h|(y)+|\nabla H|(x,y,t)h(y)\right]d\mu_{\sigma,\bar{R}}(y)\\
&=\int_{B_o(R+1)\backslash B_p(R)}\left[H(x,y,t)|\nabla h|(y)+|\nabla H|(x,y,t)h(y)\right]d\mu(y)\\
&=J_1+J_2.
\end{aligned}
\end{equation*}
By the above argument, we choose $T$ sufficiently small, then for
all $t\in(0,T)$ and all $x\in M$, $J\rightarrow 0$ as
$\bar{R}\rightarrow \infty$. Therefore Proposition \ref{integra}
holds for $T$ sufficiently small. Then the semigroup property of the
$f$-heat equation implies Proposition \ref{integra} holds for all
time $t>0$.
\end{proof}

\vspace{1 em}

Theorem \ref{Main2} leads to a uniqueness property for
$L^1$-solutions of the $f$-heat equation, which generalizes the
classical result of P. Li \cite{[Li0]}. The proof is very similar to
the one in \cite{[WuWu]}, so we omit it.

\begin{theorem}\label{Main3}
Under the same assumptions of Theorem \ref{Main2}, if $u(x,t)$ is a
nonnegative function defined on $M\times[0,+\infty)$ satisfying
\[
(\partial_t-\Delta_f)u(x,t)\leq0, \quad\int_Mu(x,t)e^{-f}dv<+\infty
\]
for all $t>0$, and
\[
\lim_{t\to 0}\int_Mu(x,t)e^{-f}dv=0,
\]
then $u(x,t)\equiv0$ for all $x\in M$ and $t\in(0,+\infty)$. In
particular, any $L_f^1$-solution of the $f$-heat equation is
uniquely determined by its initial data in $L_f^1$.
\end{theorem}


\section{eigenvalue estimate}\label{eigen}
In this section we derive eigenvalue estimates of the $f$-Laplace
operator compact smooth metric measure spaces, using the upper bound
estimate of the $f$-heat kernel and an argument of Li-Yau
\cite{[Li-Yau1]}.

When the Bakry-Emery Ricci curvature is nonnegative, we have

\begin{theorem}\label{eigenva1}
Let $(M,g,e^{-f}dv)$ be an $n$-dimensional closed smooth metric
measure space with $\mathrm{Ric}_f\geq 0$. Let
$0=\lambda_0<\lambda_1\leq\lambda_2\leq\ldots$ be eigenvalues of the
$f$-Laplacian. Then there exists a constant $C$ depending only on
$n$ and $\max_{x\in M}f(x)$, such that
\[
\lambda_k\geq \frac{C(k+1)^{2/n}}{d^2}
\]
for all $k\geq1$, where $d$ is the diameter of $M$.
\end{theorem}

\begin{proof}
Since $Ric_f\geq 0$, from Theorem 1.1, we have
\begin{equation}\label{heatestimate}
H(x,x,t)\leq\frac{C}{V_f(B_x(\sqrt{t}))},
\end{equation}
where $C$ is a constant depending only on $n$ and $B=\max_{x\in
M}f(x)$. Notice that the $f$-heat kernel can be written as
\[
H(x,y,t)=\sum^\infty_{i=0}e^{-\lambda_it}\varphi_i(x)\varphi_i(y),
\]
where $\varphi_i$ is the eigenfunction of $\Delta_f$ corresponding
to $\lambda_i$, $\|\phi_i\|_{L_f^2}=1$. By the $f$-volume comparison
theorem (see Lemma 2.1 in \cite{[WuWu]}), we get, for any $t\leq
d^2$,
\[
\frac{V_f(B_x(d))}{V_f(B_x(\sqrt{t}))}\leq
e^{4B}\left(\frac{d}{\sqrt{t}}\right)^n
\]
Taking the weighted integral on both sides of \eqref{heatestimate},
we conclude that
\begin{equation*}
\begin{aligned}
\sum^\infty_{i=0}e^{-\lambda_it}\leq
C\int_MV^{-1}_f(B_x(\sqrt{t}))d\mu \leq C\int_M p(t)d\mu,
\end{aligned}
\end{equation*}
where
\begin{equation*}
p(t)= \left\{\begin{aligned}
&e^{4B}\left(\frac{d}{\sqrt{t}}\right)^nV^{-1}_f(B_x(d)),\quad &&\mathrm{if}\quad \sqrt{t}\leq d\\
&e^{4B}V_f^{-1}(M),\quad &&\mathrm{if}\quad \sqrt{t}>d.
\end{aligned} \right.
\end{equation*}
which implies that $(k+1)e^{-\lambda_kt}\leq C q(t)$ for any $t>0$,
that is
\begin{equation}\label{heaest}
Ce^{\lambda_kt}q(t)\geq (k+1), \quad \text{for any } t>0,
\end{equation}
where
\begin{equation*}
q(t)= \left\{\begin{aligned}
&e^{4B}\left(\frac{d}{\sqrt{t}}\right)^n,\quad &&\mathrm{if}\quad \sqrt{t}\leq d\\
&e^{4B},\quad &&\mathrm{if}\quad \sqrt{t}>d.
\end{aligned} \right.
\end{equation*}

It is easy to see that  $e^{\lambda_kt}q(t)$ takes its minimum at
$t_0=\frac{n}{2\lambda_k}$. Plugging to (\ref{heaest}) we get the
lower bound for $\lambda_k$.
\end{proof}

Similarly, when the Bakry-\'Emery Ricci curvature is bounded below,
we have a similar estimate. We omit the proof since it is the same
as $\mathrm{Ric}_f\geq 0$ case.
\begin{theorem}\label{eigenva2}
Let $(M,g,e^{-f}dv)$ be an $n$-dimensional closed smooth metric
measure space with $\mathrm{Ric}_f\geq -(n-1)K$ for some constant
$K>0$. Let $0=\lambda_0<\lambda_1\leq\lambda_2\leq\ldots$ be
eigenvalues of the $f$-Laplacian. Then there exists a constant $C$
depending only on $n$ and $B=\max_{x\in M}f(x)$, such that
\[
\lambda_k\geq
\frac{C}{d^2}\left(\frac{k+1}{\exp(C\sqrt{K}d)}\right)^{\frac{2}{n+4B}}
\]
for all $k\geq1$, where $d$ is the diameter of $M$.
\end{theorem}


\section{$f$-Green's function estimate}\label{Greenf}
In this section, we will discuss the Green's function of the
$f$-Laplacian and $f$-parabolicity of smooth metric measure spaces.
It was proved by Malgrange \cite{[Ma]} that every Riemannian
manifold admits a Green's function of Laplacian. Varopoulos
\cite{[Varo]} proved that a complete manifold $(M,g)$ has a positive
Green's function only if
\begin{equation} \label{Greennecessary}
\int_1^{\infty}\frac{t}{V_p(t)}dt<\infty,
\end{equation}
where $V_p(t)$ is the volume of the geodesic ball of radius $t$ with
center at $p$. For Riemannian manifolds with nonnegative Ricci
curvature, Varopoulos \cite{[Varo]} and Li-Yau \cite{[Li-Yau1]}
proved (\ref{Greennecessary}) is the sufficient and necessary
condition for the existence of positive Green's function.

On an $n$-dimensional complete smooth metric measure space
$(M,g,e^{-f}dv)$, let $H(x,y,t)$ be a $f$-heat kernel, recall the
$f$-Green's function
\[
G(x,y)=\int^\infty_0H(x,y,t)dt
\]
if the integral on the right hand side converges. From the $f$-heat
kernel estimates, it is easy to get the following two-sided
estimates for $f$-Green's function, which is similar to Li-Yau
estimate \cite{[Li-Yau1]} of Green's function for Riemannian
manifolds with nonnegative Ricci curvauture,
\begin{theorem}\label{Green}
Let $(M,g,e^{-f}dv)$ be an $n$-dimensional complete noncompact
smooth metric measure space with $\mathrm{Ric}_f\geq 0$ and $|f|\leq
C$ for some nonnegative constant $C$. If $G(x,y)$ exists, then there
exist constants $c_1$ and $c_2$ depending only on $n$ and $C$, such
that
\begin{equation}\label{Greenest}
c_1\int^\infty_{r^2}V^{-1}_f(B_x(\sqrt{t}))dt\leq G(x,y)\leq
c_2\int^\infty_{r^2}V^{-1}_f(B_x(\sqrt{t}))dt,
\end{equation}
where $r=r(x,y)$.
\end{theorem}
As a corollary, we get a necessary and sufficient condition of the
existence of positive $f$-Green's function on smooth metric measure
spaces with nonnegative Bakry-Emery Ricci curvature and bounded
potential function,
\begin{corollary}\label{Green2} Let $(M,g,e^{-f}dv)$ be an
$n$-dimensional complete noncompact smooth metric measure space with
$\mathrm{Ric}_f\geq 0$ and $|f|\leq C$ for some nonnegative constant
$C$. There exists a positive $f$-Green's function $G(x,y)$ if and
only if
\[
\int^\infty_1V^{-1}_f(B_x(\sqrt{t}))dt<\infty.
\]
\end{corollary}

\begin{proof}[Proof of Theorem \ref{Green}]
Since $Ric_f\geq 0$ and $|f|\leq C$, Theorem \ref{Main1} holds for
any $0<t<\infty$ by letting $R\to \infty$. For the lower bound
estimate, we have
\begin{equation*}
\begin{aligned}
G(x,y)\geq\int^\infty_{r^2}H(x,y,t)dt &\geq
c_3(n,C)\int^\infty_{r^2}V^{-1}_f(B_x(\sqrt{t}))
\exp\left(\frac{-r^2}{c_4t}\right)dt\\
&\geq c_5(n,C)\int^\infty_{r^2}V^{-1}_f(B_x(\sqrt{t}))dt.
\end{aligned}
\end{equation*}
Hence the left hand side of \eqref{Greenest} follows.

For the upper bound estimate, it suffices to show that
\begin{equation}\label{guji}
\int^{r^2}_0H(x,y,t)dt \leq
c_6(n,C)\int^\infty_{r^2}V^{-1}_f(B_x(\sqrt{t}))dt.
\end{equation}
By the definition of $G$ and Theorem \ref{Main1},
\begin{equation*}
\begin{aligned}
G(x,y)=\int^\infty_0H(x,y,t)dt&=\int^{r^2}_0H(x,y,t)dt+\int^\infty_{r^2}H(x,y,t)dt\\
&\leq\int^{r^2}_0H(x,y,t)dt+c_7(n,C)\int^\infty_{r^2}V^{-1}_f(B_x(\sqrt{t}))dt\\
&\leq
c_8\int^{r^2}_0V^{-1}_f(B_x(\sqrt{t}))\exp\left(\frac{-r^2}{5t}\right)dt
+c_7\int^\infty_{r^2}V^{-1}_f(B_x(\sqrt{t}))dt,
\end{aligned}
\end{equation*}
where $c_7$ and $c_8$ depend on $n$ and $C$. Letting $s=r^4/t$,
where $r^2<s<\infty$, we get
\[
\int^{r^2}_0V^{-1}_f(B_x(\sqrt{t}))\exp\left(\frac{-r^2}{5t}\right)dt
=\int^\infty_{r^2}V^{-1}_f\left(B_x\left(\frac{r^2}{\sqrt{s}}\right)\right)
\exp\left(\frac{-s}{5r^2}\right)\frac{r^4}{s^2}ds.
\]
On the other hand, the $f$-volume comparison theorem (see Lemma 2.1
in \cite{[WuWu]}) gives
\[
V^{-1}_f\left(B_x\left(\frac{r^2}{\sqrt{s}}\right)\right)\leq
V^{-1}_f(B_x(\sqrt{s})) e^{4C}\left(\frac{s}{r^2}\right)^n.
\]
Therefore we get
\[
\int^{r^2}_0H(x,y,t)dt \leq
c_9(n,C)\int^\infty_{r^2}V^{-1}_f(B_x(\sqrt{s}))\left(\frac{s}{r^2}\right)^{n-2}
\exp\left(\frac{-s}{5r^2}\right)ds.
\]
Since the function $x^{n-2}e^{-x/5}$ is bound from above,
\eqref{guji} follows.
\end{proof}

\

Next we discuss $f$-nonparabolicity of steady Ricci solitons using a
criterion of Li-Tam \cite{[LiTa1],[LiTa2]}, and the $f$-heat kernel
for steady Gaussian Ricci soliton. A smooth metric measure space
$(M^n,g,e^{-f}dv)$ is called $f$-nonparabolic if it admits a
positive $f$-Green's function. An end, $E$, with respect to a
compact subset $\Omega\subset M$ is an unbounded connected component
of $M$. When we say that $E$ is an end, it is implicitly assumed
that $E$ is an end with respect to some compact subset
$\Omega\subset M$. O. Munteanu and J. Wang \cite{[MuWa]} proved that
if $\mathrm{Ric}_f\geq 0$, there exists at most one $f$-nonparabolic
end on $(M^n,g,e^{-f}dv)$.

First we observe that the criterion of Li-Tam \cite{[LiTa1],[LiTa2]}
can be generalized to smooth metric measure spaces,
\begin{lemma}\label{criter}
Let $(M^n,g,e^{-f}dv)$ be an $n$-dimensional complete smooth metric
measure space. There exists an $f$-Green's function $G(x,y)$ which
is smooth on $M\times M\setminus D$, where $D=\{(x,x)|x\in M\}$.
Moreover, $G(x,y)$ can be taken to be positive if and only if there
exists a positive nonconstant $f$-superharmonic function $u$ on
$M\setminus B_o(r)$ with the property that
\[
\liminf_{x\to\infty}u(x)<\inf_{x\in\partial B_o(r)}u(x).
\]
\end{lemma}

\begin{proof}[Proof of Theorem \ref{posGreen}]
Let $(M,g,f)$ be a nontrivial gradient steady soliton, we have
\[
\Delta f+R=0\quad \mathrm{and}\quad R+|\nabla f|^2=a.
\]
Chen \cite{[Chen]} proved that $R\geq 0$, so $a>0$. It was proved in
\cite{[PW],[FG]} (see also \cite{[Wu]}) that $\liminf R=0$, and
either $R>0$ or $R\equiv 0$.

By the Bochner formula, we get
\[
\Delta_f R=-2|Ric|^2\leq 0.
\]
If $R>0$ on $M$, then it is a nonconstant positive $f$-superharmonic
function, and $\lim\inf_{x\to \infty}R(x)=0$. Therefore, by Lemma
\ref{criter}, we conclude $G(x,y)$ is positive.

\

If $R\equiv 0$, then by Proposition 4.3 in \cite{[PW]}, $(M^n,g)$
splits isometrically as $(N^{n-k}\times \mathbb{R}^k, g_N+g_0)$,
where $(N^{n-k},g_N)$ is a Ricci-flat manifold, and
$(\mathbb{R}^k,g_0,f)$ is a steady Gausian Ricci soliton with
$f=\langle u,x\rangle+v$ for some $u,v\in\mathbb{R}^n$. Therefore a
$f$-Green's function on $(\mathbb{R}^k,g_0,f)$ is a $f$-Green's
function on $(M,g,f)$.

By \cite{[WuWu]}, for one-dimensional steady Gaussian Ricci soliton,
the $f$-heat kernel is given by
\[
H_{\mathbb{R}}(x,y,t)=\frac{e^{\pm \frac{x+y}{2}}\cdot
e^{-t/4}}{(4\pi t)^{1/2}}
\times\exp\left(-\frac{|x-y|^2}{4t}\right).
\]
for any $x,y\in \mathbb{R}$ and $t>0$. Therefore for any $x,y\in
\mathbb{R}$,
\[
G(x,y)=\int^\infty_0H_{\mathbb{R}}(x,y,t) dt<\infty,
\]
hence there exists a positive $f$-Green function.

For higher dimensional steady Gaussian Ricci soliton
$(\mathbb{R}^k,\ g_0,f)$, define
\[
H_{\mathbb{R}^k}(x,y,t)=H_{\mathbb{R}}(x_1,y_1,t)\times
H_{\mathbb{R}}(x_2,y_2,t)\times\ldots\times
H_{\mathbb{R}}(x_k,y_k,t),
\]
where $x=(x_1,x_2,\ldots,x_k)\in \mathbb{R}^k$,
$y=(y_1,y_2,\ldots,y_k)\in \mathbb{R}^k$, and
$H_{\mathbb{R}}(x_i,y_i,t)$ is the $f$-heat kernel for
$(\mathbb{R},g_0,u_ix_i+v_i)$. It is easy to check that
$H_{\mathbb{R}^k}(x,y,t)$ is an $f$-heat kernel on $(\mathbb{R}^k,\
g_0,f)$.

Then for any $x,y\in \mathbb{R}^k$,
\[
G(x,y)=\int^\infty_0 H_{\mathbb{R}^k}(x,y,t) dt<\infty.
\]
Therefore there exists a positive $f$-Green function on an
$k$-dimensional steady Gaussian soliton.
\end{proof}

\bibliographystyle{amsplain}

\end{document}